\documentclass[11pt]{siamart0216}

\usepackage{braket,amsfonts}

\usepackage{array}

\usepackage[caption=false]{subfig}

\usepackage{pgfplots}

\newsiamthm{claim}{Claim}
\newsiamremark{rem}{Remark}
\newsiamremark{expl}{Example}
\newsiamremark{hypothesis}{Hypothesis}
\crefname{hypothesis}{Hypothesis}{Hypotheses}

\usepackage{algorithmic}

\usepackage{graphicx,epstopdf}

\Crefname{ALC@unique}{Line}{Lines}

\usepackage{amsopn}

\usepackage{xspace}
\usepackage{bold-extra}
\usepackage[most]{tcolorbox}

\colorlet{texcscolor}{blue!50!black}
\colorlet{texemcolor}{red!70!black}
\colorlet{texpreamble}{red!70!black}
\colorlet{codebackground}{black!25!white!25}

\lstdefinestyle{siamlatex}{%
	style=tcblatex,
	texcsstyle=*\color{texcscolor},
	texcsstyle=[2]\color{texemcolor},
	keywordstyle=[2]\color{texemcolor},
	moretexcs={cref,Cref,maketitle,mathcal,text,headers,email,url},
}

\tcbset{%
	colframe=black!75!white!75,
	coltitle=white,
	colback=codebackground, 
	colbacklower=white, 
	fonttitle=\bfseries,
	arc=0pt,outer arc=0pt,
	top=1pt,bottom=1pt,left=1mm,right=1mm,middle=1mm,boxsep=1mm,
	leftrule=0.3mm,rightrule=0.3mm,toprule=0.3mm,bottomrule=0.3mm,
	listing options={style=siamlatex}
}

\newtcblisting[use counter=example]{example}[2][]{%
	title={Example~\thetcbcounter: #2},#1}

\newtcbinputlisting[use counter=example]{\examplefile}[3][]{%
	title={Example~\thetcbcounter: #2},listing file={#3},#1}

\DeclareTotalTCBox{\code}{ v O{} }
{ 
	fontupper=\ttfamily\color{black},
	nobeforeafter,
	tcbox raise base,
	colback=codebackground,colframe=white,
	top=0pt,bottom=0pt,left=0mm,right=0mm,
	leftrule=0pt,rightrule=0pt,toprule=0mm,bottomrule=0mm,
	boxsep=0.5mm,
	#2}{#1}

\patchcmd\newpage{\vfil}{}{}{}
\newtheorem{remark}[theorem]{Remark}

\renewcommand{\d}{\mathrm{d}}

\newcommand{\ep}{\varepsilon}
\newcommand{\ph}{\varphi}

\newcommand{\wt}{\widetilde}
\newcommand{\pr}{\prime}

\newcommand{\R}{\mathbb{R}}

\newcommand{\E}{\mathbb{E}}
\newcommand{\PP}{\mathbb{P}}
\newcommand{\T}{\mathcal{T}}

\newcommand{\F}{\mathcal{F}}
\newcommand{\lt}{\left}
\newcommand{\rt}{\right}

\begin{tcbverbatimwrite}{tmp_BJ-meanfield_header.tex}
\title{Mean-field SDE driven by a fractional Brownian motion and related stochastic control problem} 
\author{Rainer BUCKDAHN
\thanks{D\'epartement de Math\'ematiques, Universit\'e de Bretagne Occidentale, 29285, Brest, France.  
School of Mathematics, Shandong University, 250100, Jinan, P.R.China.
 The work of Rainer BUCKDAHN was supported by "FMJH Program Gaspard Monge in optimization and operation research", and the support of EDF to this program, also by the ANR project CAESARS (ANR-15-CE05-0024). \email{rainer.buckdahn@univ-brest.fr}}
\and
Shuai JING
\thanks{Corresponding author. School of Management Science and Engineering, Central University of Finance and Economics, 100081, Beijing, P.R.China.
The work of Shuai JING was partially supported by National Natural Science Foundation of China (NSFC Project No. 11301560, 71401188 and 71401195). \email{jing@cufe.edu.cn} }
}

\headers{MFSDE driven by FBM and related stochastic control problem}
{Rainer Buckdahn and Shuai Jing}
\end{tcbverbatimwrite}
\input{tmp_BJ-meanfield_header.tex}

\begin{document}
	\maketitle	
\begin{tcbverbatimwrite}{tmp_BJ-meanfield_abstract.tex}
\begin{abstract}
	We study a class of mean-field stochastic differential equations driven by a fractional Brownian motion with Hurst parameter $H\in(1/2,1)$ and a related stochastic control problem. We derive a Pontryagin type maximum principle and the associated adjoint mean-field backward stochastic differential equation driven by a classical Brownian motion, and we prove that under certain assumptions, which generalise the classical ones, the necessary condition for the optimality of an admissible control is also sufficient. 
\end{abstract}

\end{tcbverbatimwrite}
\input{tmp_BJ-meanfield_abstract.tex}

\section{Introduction}
In this paper we consider a class of mean-field stochastic control problem driven by a fractional Brownian motion with Hurst parameter $H\in(1/2,1)$ given by
\begin{equation}
\label{eq:control0}
X_t^u=x+\int^t_0\sigma(\PP_{X_s^u})\d B^H_s+\int^t_0 b(\PP_{(X_s^u,u_s)},X_s^u, u_s)\d s,
\end{equation}
where $x\in\R$, and $u\in\mathcal{U}([0,T])$ is an adapted control process taking values in a convex open set in $\R^m$, $\PP_{X^u_s}$ is the law of $X^u_s$ and $\PP_{(X_s^u,u_s)}$ is the joint law of $(X_s^u,u_s)$. 
Our aim is to characterise an optimal control $u^*\in \mathcal{U}([0,T])$ such that
\begin{equation}
\label{eq:optimal cost0}
J(u^*)=\inf_{u\in \mathcal{U}([0,T])}J(u).
\end{equation}
where the cost functional has the form
\begin{equation}
\label{eq:cost0}
J(u)=\E\lt[\int^T_0f\lt(\PP_{(X^u_t,u_t)},X^u_t,u_t\rt)\d t +g\lt(X_T^u,\PP_{X^u_T}\rt)\rt],
\end{equation}
for some functions $f$ and $g$ specified later.

The mean-field (or McKean-Vlasov type) stochastic differential equation (SDE) driven by classical Brownian motion was introduced by Kac \cite{Ka1} \cite{Ka2} to study the Boltzman equation and the Vlasov kinetic equation. Later Lasry and Lions \cite{LaLi} worked on mean-field stochastic games.  Henceforth the applications for mean-field problem attracted wide attention. Buckdahn et al. \cite{BDLP} \cite{BLP} studied special mean-field games and derived a kind of mean-field BSDEs associated with non local PDEs.  Carmona and Delarue \cite{CaDe} studied the existence and uniqueness of a class of  mean-field forward-backward SDEs by applying the continuation  method  proposed in Peng and Wu \cite{PeWu}. 

Stochastic control problems driven by a fractional Brownian motion also have been studied by several authors.  However, compared with the vast literatures on stochastic control problems driven by classical Brownian motion,  few has been done and there are a lot of open questions. The main reason is that fractional Brownian motion is neither a Markov process nor a semi-martingale, hence the classical methods cannot be applied directly here.  Biagini et al. \cite{BHOS} obtained a maximum principle for a stochastic control problem driven by an $m$-dimensional fractional Brownian motion with Hurst parameter $H\in(1/2, 1)^m$. For $H\in(0,1/2)$, Hu and Zhou \cite{HuZh} considered a linear stochastic optimal control problem and  obtained a Riccati equation, a BSDE driven by the fractional Brownian motion and the underlying Brownian motion. Han et al. \cite{HHS} obtained a stochastic maximum principle
for a control problem  driven by a fractional Brownian motion with $H>1/2$ and their adjoint equations is a linear BSDE again driven by the fractional Brownian motion and the underlying Brownian motion. We emphasise that their results need strong assumptions, and in particular, Malliavin differentiability of the optimal control process, which are not easily fulfilled. By applying Girsanov transformation, in \cite{BuJi} we studied a stochastic control system involving both a standard and an independent  fractional Brownian motion with Hurst parameter less than 1/2, , and we obtained as adjoint equation a BSDE driven by the Brownian motion and an independent martingale. 

In this paper, by applying Girsanov transformation, we first prove the existence and the uniqueness result for a  mean-field SDE of the form 
\begin{equation}
\label{eq:semilinear0}
X_t=\xi+\int^t_0\left(\gamma_s X_s+\sigma(s,\mathbb{P}_{(X_s,\Theta_s)}\right)\d B_s^H +\int^t_0 b(s,\mathbb{P}_{(X_s,\Theta_s)},X_s)\d s,
\end{equation}
where  $\xi$ is a square integrable random variable, $\Theta$ is a given square integrable process and $\gamma$ is a deterministic function. Then we use these results to consider a stochastic control problem with dynamics $X$ (for $\gamma=0$ and $\Theta=u$ an admissible control) and we derive the Pontryagin type maximum principle. 

We give  a necessary as well as a sufficient condition. The maximum principle leads to a  coupled system involving a mean-field forward-backward SDE, where the forward equation is a mean-field SDE driven by the fractional Brownian motion, while the backward equation is a mean-field BSDE driven only by the underlying Brownian motion, with terminal condition depending on the fractional Brownian motion. We also show that, if the time interval is small enough, the mean-field FBSDE is solvable and allows us to get an optimal control and the associated dynamics. A more general discussion of such coupled FBSDEs is foreseen for a forthcoming paper. It is worth noting that our controls are not assumed to be Malliavin differentiable.   

The paper is organised as follows: In Section 2 we give some preliminaries on fractional Brownian motion and differentiability for functions of measures. In Section 3 we study the existence and uniqueness of semi-linear  mean-field stochastic differential equations driven by a fractional Brownian motion. Our main results on the Pontryagin's maximum principle are stated in Section 4.
\section{Preliminaries}
\subsection{Fractional Brownian Motion}
Let $T>0$ be a fixed horizon. We consider a complete probability space $(\Omega,\F,\PP)$. A fractional Brownian motion $B^H=\{B^H_t, t\in[0,T]\}$ with Hurst parameter $H\in(0,1)$ is a centred Gaussian process on $(\Omega,\F,\PP)$ with covariance function
\[
R_H(t,s)=\E\lt[B^H_tB^H_s\rt]=\frac12(t^{2H}+s^{2H}-|t-s|^{2H}),\, s, t\in [0,T].
\]
For $H\in(1/2,1)$, it is
well known that the fractional Brownian motion has the representation as
follows: 
\[
B^H_t=\int^t_0K_H(t,s)\d W_s,
\]
where $W$ is a suitable Brownian motion on the space
$(\Omega, \F, \PP)$. The kernel function is given by
\[
K_H(t,s)=c_H s^{1/2-H}\int^t_s (u-s)^{H-3/2}u^{H-1/2}\d u, \quad t>s,
\]
with the constant
\[
c_H=[H(2H-1)/\beta(2-2H,H-1/2)]^{1/2},
\]
where $\beta(\alpha,\gamma)=\Gamma(\alpha+\gamma)/(\Gamma(\alpha)\Gamma(\gamma))$ is the Beta function and $\Gamma(\alpha)=\int^\infty_0x^{\alpha-1}e^{-x}\d x$ is the Gamma function. 
\subsection{Fractional Calculus}
For a detailed account on the fractional calculus theory, we refer,
for instance, to Biagini et al. \cite{BHOZ} and Samko et al. \cite{Sa}.

Let $f:[0,T]\rightarrow \R$ be a Lebesgue integrable function,
and $\alpha \in (0,1)$. The fractional Riemann Liouville integrals of $f$ are defined as follows: 

The {\it right--sided} and {\it left--sided fractional integrals} $I^\alpha_{T-}(f)(x)$ and $I^\alpha_{0+}(f)(x)$ of $f$ of order $\alpha$ are given by
$$
I^\alpha_{T-}(f)(x)=\frac{1}{\Gamma (\alpha)} \int^T_x
\frac{f(u)}{(u-x)^{1-\alpha}}\d u, \quad \hbox{\rm for almost all}\;\;x\in
[0,T].
$$
and 
$$
I^\alpha_{0+}(f)(x)=\frac{1}{\Gamma (\alpha)} \int^x_0
\frac{f(u)}{(x-u)^{1-\alpha}}\d u, \quad \hbox{\rm for almost all}\;\;x\in
[0,T].
$$

Note that $I^\alpha_{T-}(f)(x)$ and $I^\alpha_{0+}(f)(x)$ are well-defined because the Fubini theorem implies that they are functions in $L^p ([0,T])$,  $p\geq 1$,
whenever $f\in L^p ([0,T])$.

We denote by $I^\alpha_{T-}(L^p)$ (respectively, $I^\alpha_{0+}(f)(x)$), $\  p\geq 1$, the families of all
functions $f\in L^p ([0,T])$ such that
\begin{equation}
\label{eq:2.2.1} f=I^\alpha_{T-}(\ph), (\textrm{respectively}, f=I^\alpha_{0+}(\ph)),
\end{equation}
for some $\ph \in L^p([0,T])$. Samko et al. \cite{Sa} (Theorem
13.2) provide a characterization of the space $I^\alpha_{T-}(L^p)$,
$p>1$. The function $\ph$ satisfying $(\ref{eq:2.2.1})$ coincides with the {\it right--sided fractional derivative} 
\begin{equation}
\label{eq:2.2.2} (D^\alpha_{T-}f)(x)=\frac{1}{\Gamma (1-\alpha)}
\biggl(\frac{f(x)}{(T-x)^\alpha}+\alpha \int^T_x
\frac{f(x)-f(u)}{(u-x)^{1+\alpha}}\d u\biggr),
\end{equation}
respectively, the {\it left--sided functional derivative}
\begin{equation}
\label{eq:2.2.3} (D^\alpha_{0+}f)(x)=\frac{1}{\Gamma (1-\alpha)}
\biggl(\frac{f(x)}{x^\alpha}+\alpha \int^x_0
\frac{f(x)-f(u)}{(x-u)^{1+\alpha}}\d u\biggr),
\end{equation}
when the integrals are well defined. Moreover, we have
\begin{equation}
	\label{eq:2.2.4}
D_{0+}^{\alpha}f=\frac{\d}{\d x}I_{0+}^{1-\alpha} f,
\end{equation}
and
\begin{equation}
\label{eq:2.2.5}
D_{T-}^{\alpha}f=\frac{\d}{\d x}I_{T-}^{1-\alpha} f,
\end{equation}
if everything is well-defined.

Furthermore, we have the following integration by parts formula for the fractional integrals
\begin{equation}
\label{eq:2.2.6}\int^T_0I_{0+}^\alpha f(x)g(x)\d x=\int^T_0f(x)I_{T-}^\alpha g(x)\d x,
\end{equation}
if $f\in L^p[0,T]$, $g\in L^q[0,T]$, $1/p+1/q\leq 1+\alpha$. The corresponding integration by parts formula for the fractional derivatives is
\begin{equation}
\label{eq:2.2.7}\int^T_0D_{0+}^\alpha f(x)g(x)\d x=\int^T_0f(x)D_{T-}^\alpha g(x)\d x,
\end{equation}
for $f\in I^\alpha_{0+}(L^p[0,T])$, $g\in I_{T-}^\alpha (L^q[0,T])$, $1/p+1/q\leq 1+\alpha$.

\subsection{Stochastic integrals with respect to fractional Brownian motion}
Most of the results in this section can be found in Biagini et al. \cite{BHOZ}, Han et al. \cite{HHS}  and Hu \cite{Hu}.

For the kernel function $K_H(t,s)$, let $\mathcal{H}$ be the set of functions $f$ which can be represented as 
$$f(t)=\int^t_0K_H(t,s)\hat{f}(s)\d s$$
for some $\hat{f}\in L^2([0,T])$. We denote by $\mathcal{E}$ be the space of step functions on $[0,T]$ and define  $\varphi(t,s)=H(2H-1)|s-t|^{2H-2}$.   We consider the scalar product on $L^2([0,T])$:
$$
\langle f,g\rangle_H:=\int^T_0\int^T_0f(s)g(t)\varphi(t,s)\d s \d t,
$$
and we define a linear map $\mathcal{I}$ on the space $\mathcal{E}$ by
\begin{equation*}
	\begin{aligned}[c]
	 \mathcal{I}: (L^2([0,T]),\langle,\rangle_H)&\rightarrow \mathcal{H}\\
	I_{[0,t]}&\mapsto R(t,\cdot).
	\end{aligned}
\end{equation*}
Then the extension of this map to the closure of $(L^2([0,T]),\langle,\rangle_H)$ is a representation of $\mathcal{H}$. The map $\mathcal{I}$ also induces the following isometry:
\begin{equation*}
\begin{aligned}[c]
\mathcal{J}: \quad \overline{(L^2([0,T]),\langle,\rangle_H)}&\rightarrow&L^2(\Omega)\\
I_{[0,t]}&\mapsto&B^H_t.
\end{aligned}
\end{equation*} 
This allows to define the Wiener integrals with respect to $B^H$:
$$
B^H(\psi):=\mathcal{J}(\psi), \qquad \psi\in\mathcal{H}.
$$
We also use the notations $B^H(\psi)=\int^T_0\psi(t)\d B^H(t)$ and $B^H(\psi I_{[0,t]})=\int^t_0\psi(s)\d B^H_s,$ $t\in[0,T]$.

We denote by $\mathcal{S}$ the set of all polynomial functions of $B^H(\psi_j)=\int^T_0\psi_j(t)\d B^H(t)$. For an element $F\in\mathcal{S}$, having the form 
$$
F=g(B^H(\psi_1), \cdots, B^H(\psi_n)), 
$$
where $g$ is a polynomial of $n$ variables, we define its Malliavin derivative $D_s^H F$ by
$$
D^H_s F:=\sum_{i=1}^{n}\frac{\partial g}{\partial x_i}(B^H(\psi_1), \cdots, B^H(\psi_n))\psi_i(s),\qquad  0\le s \le T.
$$
For any $F\in\mathcal{S}$ as above and $p\in(0,\infty)$, we define the following norm 
$$
\|F\|_{H,1,p}:=\|F\|_p+\left[\mathbb{E}\left(\int^T_0\left|D_t^H F\right|^2\d t\right)^{p/2}\right]^{1/p}.
$$
We denote by $\mathbb{D}_{H,1,p}$ the Banach space obtained by completing $\mathcal{S}$ with respect to the norm $\|\cdot\|_{H,1,p}$. 

The classical Malliavin derivative $D^W$ with respect to the underlying Brownian motion $W$ and the space $\mathbb{D}_{1,p}^W$ can be defined in a similar and classical way, which we omit here.

We define an operator $K_H$ on $\mathcal{H}$ as:
\[
(K_H\psi)(s)=c_H\Gamma(H-1/2)s^{1/2-H}I_{0+}^{H-1/2}(u^{H-1/2}\psi(u))(s).
\] 
Then its adjoint operator $K^*_H$  on $\mathcal{H}$ is:
\[
(K^*_H\psi)(s)=c_H\Gamma(H-1/2)s^{1/2-H}I_{T-}^{H-1/2}(u^{H-1/2}\psi(u))(s),
\]
and  its inverse operator ${K^*_H}^{-1}$ is:
\[
({K^*_H}^{-1}\psi)(s)=\frac{1}{c_H\Gamma(H-1/2)}s^{1/2-H}(D_{T-}^{H-1/2}u^{H-1/2}\psi(u))(s).
\]
For $\psi\in\mathcal{H}$, the following relationship holds:
\[
\int^T_0 \psi(t)\d B^H(t)=\int^T_0({K}_H^*\psi)(t)\d W(t)
\]
and
\[
\int^T_0 \psi(t)\d W(t)=\int^T_0({{K}_H^*}^{-1}\psi)(t)\d B^H(t).
\]

Therefore, if we denote by $\mathbb{F}=\{\mathcal{F}_t, t\in[0,T]\}$ the filtration generated by the fractional Brownian motion $\{B^H_t\}_{t\in[0,T]}$, it coincides with the one generated by the underlying Brownian motion $\{W_t, t\in[0,T]\}$.

We have the following proposition (see also Proposition 5.2.1 in Nualart \cite{Nu}):
\begin{proposition}
	\label{prop_DhDw}
	If $F\in\mathbb{D}^W_{1,2}\bigcap\mathbb{D}_{H,1,2}$, then
	$$D^H_s F={K_H^*}^{-1} D^W_s F. $$
\end{proposition}

However, it is more convenience for fractional Brownian motions to use another Malliavin derivative, which is defined as
\begin{equation}
\label{def_DDH}
\mathbb{D}_s^H F=\int^T_0\varphi(s-r)D_r^H F\d r,
\end{equation}
where 
$$
\varphi(r)=H(2H-1)|r|^{2H-2}, \qquad 0\le r\le T.
$$
From Section 5.8 in \cite{Hu} we know
$$
\mathbb{D}^H_s F=K_HK_H^* D^H_sF.
$$

Now we can define  by the following result the more general Skorohod type integral $\int^T_0 f(t)\d B^H_t$ as the divergence operator related to $\mathbb{D}_t^H$ (See, for example Theorem 6.23 in \cite{Hu}, or Proposition 2.3 in \cite{HHS}).
\begin{definition}
\label{prop_duality}
Let $f: ([0,T]\times\Omega,\mathcal{B}([0,T])\otimes\mathcal{F})\rightarrow (\mathbb{R},\mathcal{B}(\R))$ be a jointly measurable square integrable process. We say that $f$ is integrable with respect to $B^H$ ($f\in \mathrm{Dom} (\delta_H)$ ), if there is some $\delta_H (f)\in L^2(\Omega, \mathcal{F},\PP)$ such that for all $G\in \mathbb{D}_{H,1,2}$,  
\begin{equation}
\label{eq_duality}
\mathbb{E}\left[G\delta_H(f)\right]=\int^T_0\mathbb{E}\left[f(t)\mathbb{D}_t^H G\right]\d t.
\end{equation}
\end{definition}
If $fI_{[s,t]}\in \mathrm{Dom}(\delta_H),$ we write $\int^t_s f(r)\d B^H_r:=\delta_H(fI_{[s,t]})$,  $s,t\in[0,T]$.

From the classical Malliavin calculus theory (refer to, Nualart \cite{Nu} and Buckdahn \cite{Bu}), we have the following proposition (see also Proposition 6.25 in Hu \cite{Hu}).
\begin{proposition} \label{prop_malliavin}
If $f\in\mathrm{Dom}(\delta_H)$, it holds that:
\begin{equation}
\label{eq:skorohodsquare}
\begin{aligned}
\E\lt[\lt|\int^T_0f(t)\d B_t^H\rt|^2\rt]=\E\lt[\int^T_0|K^*_Hf(t)|^2\d t\rt]+2\E\lt[\int^T_0\int^s_0\mathbb{D}_s^H f(r) \mathbb{D}^H_rf(s)\d r\d s\rt].
\end{aligned}
\end{equation}
\end{proposition}

The Stratonovich integral with respect to fractional Brownian motion can be defined from the Skorohod integral as follows (see  Theorem 3.9 in \cite{DHP}).
\begin{proposition}
	\label{def_stratonovich}
	Let $f:\Omega\times[0,T]\to\R$ be a stochastic process which is Malliavin differentiable such that the following holds:
	$$
	\E\lt[\int^T_0\int^T_0|f(s)f(t)|\varphi(s-t)\d s\d t+\int^T_0\int^T_0|\mathbb{D}^H_s f(t)|^2\d s \d t\rt]<\infty.
	$$
	Then the Stratonovich integral $\int^T_0f(t)\d^\circ B^H_t$ exists and 
	\begin{equation}
	\label{eq:def_Stato}
	\int^T_0f(t)\d^\circ B^H_t=\int^T_0f(t)\d B_t^H+\int^T_0\mathbb{D}^H_tf(t)\d t.
	\end{equation}
\end{proposition}

The following proposition can be derived from Remark 2.7.4 in Mishura \cite{Mi}.
\begin{proposition}
\label{prop_ito_stratonovich}
For $t\in[0,T]$, let  $F_1(t)=\int^t_0f_1(s)\d s+\int^t_0f_2(s)\d^\circ B^H_s$ and $G_1(t)$$=\int^t_0g_1(s)\d s+\int^t_0 g(s)\d W_s$, where $f_1, g_1$ are integrable processes,  $f_2$ satisfies the conditions in Proposition \ref{def_stratonovich} and $g_2$ is continuous square integrable adapted process. Then 
we have
\begin{equation}
\label{eq:ito_strato}
\d F_1(t)G_1(t)=F_1(t)g_2(t)\d W_t+G_1(t)f_2(t)\d ^\circ B_t^H +[F_1(t)g_1(t)+G_1(t)f_1(t)]\d t. 
\end{equation}
\end{proposition}

Combining Proposition \ref{def_stratonovich} and \ref{prop_ito_stratonovich}, it is easy to deduce the following result. 
\begin{corollary}
	\label{corol}
	For $t\in[0,T]$, let now $F(t)=\int^t_0f_1(s)\d s+\int^t_0f_1(s)\d B^H_s$ and $G(t)=\int^t_0g_1(s)\d s+\int^t_0 g_2(s)\d W_s$, where $f_1, g_1$ are integrable processes,  $f_2$ satisfies the conditions in Proposition \ref{def_stratonovich} and $g_2$ is continuous square integrable adapted process.  Then 
	we have
	\begin{equation}
	\label{eq:ito_WH}
	\d F(t)G(t)=F(t)g_2(t)\d W_t+G(t)f_2(t)\d  B_t^H +f_2(t)\mathbb{D}^H_tG(t)\d t. 
	\end{equation}
\end{corollary}
\subsection{Girsanov Transformation}
Let $\{\gamma(s), s\in[0,T]\}$ be a bounded function in $\mathcal{H}$. For any $\omega\in\Omega$, we define the following operators:
\begin{equation}
	\label{transT}
	\mathcal{T}_t(\omega)=\omega+\int^{t\wedge\cdot}_0K^*_H(\gamma I_{[0,t]})(s)\d s,
\end{equation}
and
\begin{equation}
\label{transA}
\mathcal{A}_t(\omega)=\omega-\int^{t\wedge\cdot}_0K^*_H(\gamma I_{[0,t]})(s)\d s, \, t\in[0,T].
\end{equation}
It is clear that $\mathcal{A}_t\mathcal{T}_t(\omega)=\mathcal{T}_t\mathcal{A}_t(\omega)=\omega$. Moreover, for any $F\in\mathcal{S}$, we have from the Girsanov theorem (we refer to \cite{Bu}),
\begin{equation}
\label{eq:Girsanov}
\mathbb{E}[F]=\E[F(\mathcal{T}_t)\varepsilon_t^{-1}(\mathcal{T}_t)]=\E[F(\mathcal{A}_t)\varepsilon_t],
\end{equation}
where 
\begin{equation*}
\begin{aligned}
\varepsilon_t&=\exp\left\{\int^t_0\gamma_s\d B^H_s-\frac12\int^t_0\left(K^*_H(\gamma I_{[0,t]})\right)^2(s)\d s\right\}\\
&=\exp\lt\{\int^t_0 K^*_H(\gamma I_{[0,t]})(s)\d W_s-\frac12\int^t_0\left(K^*_H(\gamma I_{[0,t]})\right)^2(s)\d s\rt\},
\end{aligned}
\end{equation*}
and hence
\begin{equation*}
\begin{aligned}
\varepsilon^{-1}_t(\T_t)&=\exp\left\{-\int^t_0\gamma_s\d B^H_s-\frac12\int^t_0\left(K^*_H(\gamma I_{[0,t]})\right)^2(s)\d s\right\}.
\end{aligned}
\end{equation*}

Following a similar argument in Lemma 2.4 in \cite{JL} , we verify that 
$$
\E\left[\sup_{t\in[0,T]}\varepsilon_t^p\right]< +\infty \quad \textrm{and} \quad \E\left[\sup_{t\in[0,T]}\varepsilon_t^p(\T_t)\right]< +\infty,\qquad \textrm{for\,\ all} \quad p\in\R.
$$ 

\subsection{Differentiability of Functions of Measures} 
Let $\mathcal{P}(\R^n)$ be the space of all probability measures on $(\R^n, \mathcal{B}(\R^n))$. We denote by $\mathcal{P}_p(\R^n)$ the subspace of  $\mathcal{P}(\R^n)$ of order $p$, which means
\[
\mathcal{P}_p(\R^n)=\{\mu\in\mathcal{P}(\R^n):\int_{\R^n} |x|^p\mu(\d x)<+\infty\}.
\]
On $\mathcal{P}_p(\R^{n})$, the Wasserstein metric of order $p$  is defined by
\[
\begin{aligned}
W_p(\mu,\nu)=\inf\bigg\{\lt(\int_{\R^{2n}} |x-y|^p\rho(\d x,\d y)\rt)^{\frac1p},& \ \rho\in\mathcal{P}_i(\R^{2n})\ \hbox{such that}\\ &\rho(\cdot\times\R^n)=\mu \ \hbox{and}\ \rho(\R^n\times\cdot)=\nu\bigg\}.
\end{aligned}
\]
In this paper, we will use Wasserstein metrics of order 1 and 2: $W_1$ and  $W_2$.
Notice that if $\xi$ and $\eta$ are two $p$-integrable random variables with laws $\PP_\xi$ and $\PP_\eta$, then we have $W_p(\PP_\xi,\PP_\eta)\le\lt(\E{|\xi-\eta|^p}\rt)^{\frac1p}$ since we can choose a special $\rho=\PP_{(\xi,\eta)}$ in the above definition. In this paper, the notion of differentiability for functions of measures we use is that introduced by P.~L.~Lions in his course at the Coll\`ege de France and summarized by Cardaliaguet \cite{Car}. We also refer to Carmona and Delarue \cite{CaDe}.

Notice that, as $(\Omega, \F, \PP)$ carries a fractional Brownian motion, it is rich enough in the sense that $\mathcal{P}_2(\R^n)=\{P_\xi, \xi\in L^2(\Omega,\F,\PP; \R^n)\}, n\ge 1.$

Given a function $\sigma: \mathcal{P}_2(\R)\to \R$, for any random variable $\xi\in L^2(\Omega,\F,\PP)$, we set $\tilde{\sigma}(\xi):=\sigma(\PP_{\xi})$.
\begin{definition}
The function $\sigma$ is said to be differentiable at $\mu\in \mathcal{P}_2(\R)$, if there exists a random variable $\tilde{\xi}\in L^2(\Omega,\F,\PP)$ with $\PP_{\tilde{\xi}}=\mu$ such that $\tilde{\sigma}: L^2(\Omega,\F,\PP)\to \R$ is Fr\'echet differentiable at $\tilde{\xi}$.
\end{definition}

For simplicity, we suppose that $\tilde{\sigma}:  L^2(\Omega,\F,\PP)\to \R$ is Fr\'echet differentiable. We denote its Fr\'echet derivative at $\tilde{\xi}$ by $D\tilde{\sigma}(\tilde{\xi})$. Notice that $D\tilde{\sigma}(\tilde{\xi}):L^2(\Omega,\F,\PP)\to \R$ is a continuous linear mapping; we write $D\tilde{\sigma}(\tilde{\xi})\in L(L^2(\Omega,\F,\PP),\R)$. Hence,
\[
\sigma(\PP_{\xi})-\sigma(\PP_{\tilde{\xi}})=\tilde{\sigma}(\xi)-\tilde{\sigma}(\tilde{\xi})=
\langle(D\tilde{\sigma})({\xi}),(\xi-\tilde{\xi})\rangle_{L^2}+o(|\xi-\tilde{\xi}|_{L^2}),\  \textrm{as}\ |\xi-\tilde{\xi}|_{L^2}\to 0.
\]

According to Cardaliaguet \cite{Car}, with the Riesz representation theorem, $D\tilde{\sigma}(\xi)\in $ $L(L^2(\Omega,\F,$ $\PP),\R)\equiv L^2(\Omega,\F,\PP)$, i.e., there exists a random variable $\theta\in L^2(\Omega,\F,\PP)$ such that $D\tilde{\sigma}(\xi)(\eta)=E[\theta\eta]$, for an $\eta\in L^2(\Omega,\F,\PP)$. Due to the by now well known result by P.-L.Lions, there is a Borel function $h_{\PP_\xi}:\R\to\R$, such that $\theta=h_{\PP_\xi}(\xi)$, $\PP$-a.s. 

We define the derivative of $\sigma$ with respect to the measure at $\PP_\xi$ by putting $\partial_\mu\sigma(\PP_\xi,x)=h_{\PP_\xi}(x)$. Notice that $\partial_\mu\sigma(\PP_\xi,x)$ is defined only $\PP_\xi(\d x)$-a.e. uniquely. Therefore,
\[
\sigma(\PP_{\tilde{\xi}})-\sigma(\PP_{\xi})=\E[
\partial_\mu\sigma(\PP_\xi,\xi)(\tilde{\xi}-\xi)]+o(|\tilde{\xi}-\xi|_{L^2}),\  \textrm{as}\ |\xi-\tilde{\xi}|_{L^2}\to 0. 
\]

For example, if, for $\xi\in L^2(\Omega,\F,\PP)$ and $\sigma,\varphi\in C^1_b(\mathbb{R})$, we consider $\sigma(\PP_\xi)=\sigma(\E[\varphi(\xi)])$, we have $\tilde\sigma(\xi)=\sigma(\E[\varphi(\xi)])$, and a straight forward computation shows 
\[
\begin{aligned}
D\tilde{\sigma}(\xi)(\eta)
=\E[\sigma^\prime(\E[\varphi(\xi)])\varphi^\prime(\xi)\eta],\, \textrm{for all}\, \eta\in L^2(\Omega,\F,\PP),
\end{aligned}
\]
i.e.,  $\partial_\mu\sigma(\PP_\xi,x)=\sigma^\prime(\E[\varphi(\xi)])\varphi^\prime(x)$.

As concerns the well-definedness of the derivative $\partial_\mu\sigma (\mathbb P_\xi, x):=h_{\mathbb P_\xi} (x)$, i.e.,  the dependence of $h_{\mathbb P_\xi}$ on $\xi$ only through $\PP_\xi$, it can be shown by a rather simple argument: Let $\xi_1,\xi_2\in L^2(\Omega,{\cal F},\PP)$ be such that $\widetilde{\sigma}$  is differentiable at both $\xi_1$ and $\xi_2$ and  $\PP_\xi=\PP_{\xi_1}=\PP_{\xi_2}$. 
Then, for any bounded Borel function $\phi:\R\rightarrow \R$, for $i=1,2,$
$$\E[h_{\PP_{\xi_i}}(\xi_i)\phi(\xi_i)]=D\tilde{\sigma}(\xi_i)(\phi(\xi_i))=\partial_\varepsilon\widetilde{\sigma}(\xi_i+\varepsilon\phi(\xi_i))_{|\varepsilon=0}=\partial_\varepsilon\sigma(P_{\xi_i+\varepsilon\phi(\xi_i)})_{|\varepsilon=0}.$$
But as $\PP_{\xi_1}=\PP_{\xi_2}$, also $\PP_{\xi_1+\varepsilon\phi(\xi_1)}=\PP_{\xi_2+\varepsilon\phi(\xi_2)},$ for all $\varepsilon>0$. This implies that, as $h_{\PP_{\xi_i}}$ is deterministic, 
$$\E\lt[h_{\PP_{\xi_1}}(\xi_1)\phi(\xi_1)\rt]=\E\lt[h_{\PP_{\xi_2}}(\xi_2)\phi(\xi_2)\rt]=\E\lt[h_{\PP_{\xi_2}}(\xi_1)\phi(\xi_1)\rt],$$ 
for all bounded Borel function $\phi$. Finally, choosing $\phi(x)=$sign$(h_{\xi_1}(x)-h_{\xi_2}(x)),\ x\in \R$, we get 
$$\E[|h_{\xi_1}(\xi_1)-h_{\xi_2}(\xi_1)|]=\E[(h_{\xi_1}(\xi_1)-h_{\xi_2}(\xi_1))\phi(\xi_1)]=0,$$ 
i.e., $h_{\xi_1}(x)=h_{\xi_2}(x),\, \PP_{\xi_1}(=\PP_{\xi_2})$-a.s. 

\smallskip
In the last part of this paper, we need the joint convexity of a function on $(\R^n\times \mathcal{P}_2(\R^d))$. A differentiable function $g$ defined on  $(\R^n\times \mathcal{P}_2(\R^d))$ is convex, if for every $(x,\mu)$ and $(x^\pr,\mu^\pr)\in  (\R^n\times \mathcal{P}_2(\R^d))$, we have 
\begin{equation}
\label{convexity}
g(x^\pr,\mu^\pr)-g(x,\mu)-\langle\partial_x g(x,\mu), (x^\pr-x)\rangle-\wt{\E}\lt[\langle\partial_\mu g(x,\mu)(\wt{X}), \wt{X}^\pr-\wt{X}\rangle\rt]\ge 0, 
\end{equation}
where $\wt{X}, \wt{X}^\pr\in L^2(\Omega,\F,\PP;\R^d)$ with $\PP_{\wt{X}}=\mu$ and $\PP_{\wt{X}^\pr}=\mu^\pr$, and $\langle\cdot,\cdot\rangle$ stands for the scalar product in $\R^m$, $m\in \mathbb{N}$.

Moreover,  a differentiable function $g$ defined on  $(\R^n\times \mathcal{P}_2(\R^d))$ is strictly convex, if there exists $\lambda>0$, for every $(x,\mu)$ and $(x^\pr,\mu^\pr)\in  (\R^n\times \mathcal{P}_2(\R^d))$, we have 
\begin{equation}
\label{strictconvexity}
\lambda (|x-x^\pr|^2+\E\lt[|X-X^\pr|^2\rt])\le \langle\partial_x g(x,\mu), (x^\pr-x)\rangle+{\E}\lt[\langle\partial_\mu g(x,\mu)({X}), {X}^\pr-{X}\rangle\rt], 
\end{equation}
where ${X}, {X^\pr}\in L^2(\Omega,\F,\PP;\R^d)$ with $\PP_{{X}}=\mu$ and $\PP_{{X}^\pr}=\mu^\pr$.
\section{Mean-field SDE driven by fractional Brownian motion}
\label{section sde}
In this section, we will study a class of semi-linear stochastic differential equations driven by a fractional Brownian motion. In the following sections, the constant $C$ can vary from line to line. 

Given an arbitrary square integrable process $\Theta=(\Theta_s)$ with values in $\R^m$, $m\ge 1$, let us consider the following equation:
\begin{equation}
\label{eq:semilinear}
X_t=\xi+\int^t_0\left(\gamma_s X_s+\sigma(s,\mathbb{P}_{(X_s,\Theta_s)}\right)\d B_s^H +\int^t_0 b(s,\mathbb{P}_{(X_s,\Theta_s)},X_s)\d s,
\end{equation}
where $\xi\in L^2(\Omega,\mathcal{F}_0,\PP;\R)$ and the coefficients $\sigma: [0,T]\times\mathcal{P}_2(\R \times U)\to \R$ and $b:\Omega\times [0,T]\times\mathcal{P}_2(\R \times U)\times\R\to\R$  satisfy the following conditions: 

\textbf{(H1)} For any $s\in[0,T]$, $x, x^\prime\in\R$, $\eta, \eta^\prime\in L^2(\Omega,\mathcal{F}, \PP;\R)$ and $\Theta\in L^2(\Omega,\mathcal{F}, \PP;\R^m)$, there exists a constant $C>0$ such that
$$
|\sigma(s,\PP_{(\eta,\Theta)})|\le C,
$$ 
$$
|b(s,\PP_{(\eta,\Theta)},x)|\le C\lt(1+W_1(\PP_{(\eta,\Theta)},\PP_{(0,\Theta)})+|x|\rt),
$$
$$ 
\lt|\sigma(s, \PP_{(\eta,\Theta)})-\sigma(s,\PP_{(\eta^\prime,\Theta)})\rt|\le C W_1(\PP_{(\eta,\Theta)},\PP_{(\eta^\pr,\Theta)}), 
$$
$$
|b(s,\PP_{(\eta,\Theta)},x)-b(s,\PP_{(\eta^\prime,\Theta)},x^\prime)|\le C \lt(W_1(\PP_{(\eta,\Theta)},\PP_{(\eta^\pr,\Theta)})+|x-x^\prime|\rt).
$$
\begin{remark}
It is easy to deduce from \textbf{(H1)} the following conditions:
$$
|b(s,\PP_{(\eta,\Theta)},x)|\le C\lt(1+\E\lt[|\eta|\rt]
+|x|\rt),
$$
$$ 
\lt|\sigma(s, \PP_{(\eta,\Theta)})-\sigma(s,\PP_{(\eta^\prime,\Theta)})\rt|\le C \E\lt[|\eta-\eta^\prime|\rt], 
$$
$$
|b(s,\PP_{(\eta,\Theta)},x)-b(s,\PP_{(\eta^\prime,\Theta)},x^\prime)|\le C \lt(\E\lt[|\eta-\eta^\prime|\rt]+|x-x^\prime|\rt).
$$
\end{remark}

We denote by $L^{2,*}([0,T];\R)$ the Banach space of $\mathbb{F}$-adapted process $\{\varphi(t),t\in[0,T]\}$ such that 
$$
\sup_{t\in[0,T]}\E\lt[|\varphi(t)|^2\ep^{-1}_t\rt]<+\infty.
$$ 
\begin{definition}
A solution of equation \eqref{eq:semilinear} is a stochastic process $X=(X_t)_{t\ge0}\in L^{2,*}([0,T];\R)$ such that  $XI_{[0,t]}\in\mathrm{Dom}(\delta_H), t\in[0,T]$ and equation \eqref{eq:semilinear} holds true $\PP$-a.s.
\end{definition}
\begin{remark}
	Note that for $X\in L^{2,*}([0,T];\R)$, $\big(\sigma(s,\PP_{(X,\Theta)})\big)_{s\in[0,T]}\in L^\infty([0,T])$ is a bounded and hence, square integrable deterministic function, which implies that $\int^t_0\sigma(s,\PP_{(X_s,\Theta_s)})\d B^H_s$ is well defined.
\end{remark}

To solve the equation \eqref{eq:semilinear}, we first transform it to another one. Indeed, we have the following statement.
\begin{theorem}
	\label{thm_semi-girsemi}
	Assume $X\in L^{2,*}([0,T];\R)$. Then $X$ is a solution of \eqref{eq:semilinear} if and only if it solves the following equation:
	\begin{equation}
	\label{eq:Gir-semilinear}
	\begin{aligned}
	&X_t(\mathcal{T}_t)\varepsilon_t^{-1}(\mathcal{T}_t)\\
	=&\xi+\int^t_0\sigma(s,\mathbb{P}_{(X_s,\Theta_s)})\varepsilon_s^{-1}(\mathcal{T}_s)\d B_s^H +\int^t_0 b(s,\mathcal{T}_s,\mathbb{P}_{(X_s,\Theta_s)},X_s(\mathcal{T}_s))\varepsilon_s^{-1}(\mathcal{T}_s)\d s.
	\end{aligned}
		\end{equation}
\end{theorem}

\begin{remark}
We note that for any $X\in L^{2,*}([0,T];\R)$, the expression $$
\int^t_0\sigma(s,\PP_{(X_s,\Theta_s)})\ep^{-1}_s(\T_s)\d B^H_s, 0\le t\le T,
$$ 
is well defined. Indeed, $\sigma(s,\PP_{(X_s,\Theta_s)})$ is a deterministic bounded function, and we have the following statement:
\end{remark}
\begin{lemma}
	For all $\Theta\in L^\infty([0,T])$, the process $\big(\Theta_s\ep_s^{-1}(\T_s)\big)_{s\in[0,T]}\in\mathrm{Dom}(\delta_H)$.
\end{lemma}

\begin{proof} (of Theorem \ref{thm_semi-girsemi}).
	Suppose $\{X_t, t\in[0,T]\}\in L^{2,*}([0,T];\R)$ is a solution of equation (\ref{eq:semilinear}), and that, in particular $\gamma X I_{[0,t]}\in \mathrm{Dom} (\delta_H)$, $t\in[0,T]$. Then, for any $F\in\mathcal{S}$, we have
	\begin{equation*}
	\begin{aligned}
	&\E\left[FX_t(\mathcal{T}_t)\varepsilon_t^{-1}(\mathcal{T}_t)-F\xi\right]=\E[F(\mathcal{A}_t)X_t-F\xi]\\
	=&\E\lt[F(\mathcal{A}_t)\xi-F\xi\rt]+\E\lt[F(\mathcal{A}_t)\int^t_0(\gamma_s X_s+\sigma(s,\mathbb{P}_{(X_s,\Theta_s)}))\d B_s^H\rt]\\
	&+\E\lt[F(\mathcal{A}_t)\int^t_0b(s,\mathbb{P}_{(X_s,\Theta_s)},X_s)\d s\rt]\\
	=&\E\lt[\xi\int^t_0\frac{\d F(\mathcal{A}_s)}{\d s}\d s\rt]+\E\lt[F(\mathcal{A}_t)\int^t_0(\gamma_s X_s+\sigma(s,\mathbb{P}_{(X_s,\Theta_s)}))\d B_s^H\rt]\\
	&+\E\lt[F(\mathcal{A}_t)\int^t_0b(s,\mathbb{P}_{(X_s,\Theta_s)},X_s)\d s\rt].\\
	\end{aligned}
	\end{equation*}
	We remark that $\frac{\d F(\mathcal{A}_s)}{\d s}=-\gamma_sK_HK_H^* D^H_sF(\mathcal{A}_s)=-\gamma_s\mathbb{D}_sF(\mathcal{A}_s)$. Thus, from Proposition \ref{prop_duality} we have
	\begin{equation*}
	\begin{aligned}
	&\E\left[FX_t(\mathcal{T}_t)\varepsilon_t^{-1}(\mathcal{T}_t)-F\xi\right]\\=&-\E\lt[\xi\int^t_0\gamma_s\mathbb{D}^H_sF(\mathcal{A}_s)\d s\rt]+\E\lt[\int^t_0(\gamma_sX_s+\sigma(s,\mathbb{P}_{(X_s,\Theta_s)}))\mathbb{D}^H_sF(\mathcal{A}_t)\d s\rt]\\
	&+\E\lt[\int^t_0b(s,\mathbb{P}_{(X_s,\Theta_s)},X_s)F(\mathcal{A}_t)\d s\rt].
	\end{aligned}
	\end{equation*}
	Using again that $F(\mathcal{A}_t)=F(\mathcal{A}_s)-\int^t_s\gamma_r\mathbb{D}^H_r(F(\mathcal{A}_r))\d r$, we see that 
	$$
	\mathbb{D}^H_s F(\mathcal{A}_t)=\mathbb{D}^H_sF(\mathcal{A}_s)-\int^t_s\gamma_r\mathbb{D}^H_s(\mathbb{D}^H_r(F(\mathcal{A}_r)))\d r. 
	$$
	Consequently,
	\begin{equation*}
	\begin{aligned}
	&\E\left[FX_t(\mathcal{T}_t)\varepsilon_t^{-1}(\mathcal{T}_t)-F\xi\right]\\
	=&-\E\lt[\xi\int^t_0\gamma_s\mathbb{D}^H_sF(\mathcal{A}_s)\d s\rt]+\E\lt[\int^t_0(\gamma_sX_s+\sigma(s,\mathbb{P}_{(X_s,\Theta_s)}))\mathbb{D}^H_sF(\mathcal{A}_s)\d s\rt]\\
	&-\E\lt[\int^t_0\int_s^t\gamma_r\mathbb{D}^H_s(\mathbb{D}^H_r(F(\mathcal{A}_r)))(\gamma_sX_s+\sigma(s,\mathbb{P}_{(X_s,\Theta_s)}))\d r\d s\rt]\\
	&+\E\lt[\int^t_0b(s,\mathbb{P}_{(X_s,\Theta_s)},X_s)F(\mathcal{A}_s)\d s\rt]\\
	&-\E\lt[\int^t_0\int_s^tb(s,\mathbb{P}_{(X_s,\Theta_s)},X_s)\gamma_r\mathbb{D}^H_rF(\mathcal{A}_r)\d r\d s\rt],
	\end{aligned}
	\end{equation*}
	and the Fubini theorem then yields
	\begin{equation*}
	\begin{aligned}
	&\E\left[FX_t(\mathcal{T}_t)\varepsilon_t^{-1}(\mathcal{T}_t)-F\xi\right]\\=&-\E\lt[\xi\int^t_0\gamma_s\mathbb{D}^H_sF(\mathcal{A}_s)\d s\rt]+\E\lt[\int^t_0(\gamma_sX_s+\sigma(s,\mathbb{P}_{(X_s,\Theta_s)}))\mathbb{D}^H_sF(\mathcal{A}_s)\d s\rt]\\
	&-\E\lt[\int^t_0\int_0^r\gamma_r\mathbb{D}^H_s(\mathbb{D}^H_r(F(\mathcal{A}_r)))(\gamma_sX_s+\sigma(s,\mathbb{P}_{(X_s,\Theta_s)}))\d s\d r\rt]\\
	&+\E\lt[\int^t_0b(s,\mathbb{P}_{(X_s,\Theta_s)},X_s)F(\mathcal{A}_s)\d s\rt]\\
	&-\E\lt[\int^t_0\int_0^rb(s,\mathbb{P}_{(X_s,\Theta_s)},X_s)\gamma_r\mathbb{D}^H_rF(\mathcal{A}_r)\d s\d r\rt].\\
	\end{aligned}
	\end{equation*}
Applying Proposition \ref{prop_duality} again, combined with the Fubini theorem, we have
	\begin{equation*}
	\begin{aligned}
	&\E\lt[\int^t_0\int_0^r\gamma_r\mathbb{D}^H_s(\mathbb{D}^H_r(F(\mathcal{A}_r)))(\gamma_sX_s+\sigma(s,\mathbb{P}_{(X_s,\Theta_s)}))\d s\d r\rt]\\
	=&\E\lt[\int^t_0\gamma_s\mathbb{D}^H_s(F(\mathcal{A}_s))\int^s_0(\gamma_rX_r+\sigma(r,\mathbb{P}_{(X_r,\Theta_r)}))\d B^H_r\d s\rt].
	\end{aligned}
	\end{equation*}
 Hence,
	\begin{equation*}
	\begin{aligned}
	&\E\left[FX_t(\mathcal{T}_t)\varepsilon_t^{-1}(\mathcal{T}_t)-F\xi\right]\\
	=&-\E\lt[\xi\int^t_0\gamma_s\mathbb{D}^H_sF(\mathcal{A}_s)\d s\rt]+\E\lt[\int^t_0(\gamma_sX_s+\sigma(s,\mathbb{P}_{(X_s,\Theta_s)}))\mathbb{D}^H_sF(\mathcal{A}_s)\d s\rt]\\
	&-\E\lt[\int^t_0\gamma_s\mathbb{D}^H_s(F(\mathcal{A}_s))\int^s_0(\gamma_rX_r+\sigma(r,\mathbb{P}_{(X_r,\Theta_r)}))\d B^H_r\d s\rt]\\
	&+\E\lt[\int^t_0(b(s,\mathbb{P}_{(X_s,\Theta_s)},X_s))F(\mathcal{A}_s)\d s\rt]\\
	&-\E\lt[\int^t_0\int_0^sb(r,\mathbb{P}_{(X_r,\Theta_r)},X_r)\gamma_s\mathbb{D}^H_sF(\mathcal{A}_s)\d r\d s\rt].
	\end{aligned}
	\end{equation*}
Therefore, 
    \begin{equation*}
	\begin{aligned}
	&\E\left[FX_t(\mathcal{T}_t)\varepsilon_t^{-1}(\mathcal{T}_t)-F\xi\right]\\
	=&\E\bigg[\int^t_0\gamma_s\mathbb{D}^H_sF(\mathcal{A}_s)\bigg(-\xi+X_s-\int^s_0(\gamma_rX_r+\sigma(r,\mathbb{P}_{(X_r,\Theta_r)}))\d B^H_r\\
	&\qquad-\int_0^s(b(r,\mathbb{P}_{(X_r,\Theta_r)},X_r))\d r\bigg)\d s\bigg]\\
	&+\E\lt[\int^t_0\sigma(s,\mathbb{P}_{(X_s,\Theta_s)})\mathbb{D}^H_sF(\mathcal{A}_s)\d s\rt]+\E\lt[\int^t_0b(s,\mathbb{P}_{(X_s,\Theta_s)},X_s)F(\mathcal{A}_s)\d s\rt]\\
	=&\E\lt[\int^t_0\sigma(s,\mathbb{P}_{(X_s,\Theta_s)})\mathbb{D}^H_sF(\mathcal{A}_s)\d s\rt]+\E\lt[\int^t_0b(s,\mathbb{P}_{(X_s,\Theta_s)},X_s)F(\mathcal{A}_s)\d s\rt],
    \end{aligned}
    \end{equation*}
	where we have used that $X$ solves \eqref{eq:semilinear}. Thus,
    Girsanov transformation yields
	\begin{equation*}
	\begin{aligned}
	&\E\left[F\lt(X_t(\mathcal{T}_t)\varepsilon_t^{-1}(\mathcal{T}_t)-\xi\rt)\right]\\
	=&\E\lt[\int^t_0\sigma(s,\mathbb{P}_{(X_s,\Theta_s)})\varepsilon_s^{-1}(\mathcal{T}_s)\mathbb{D}^H_sF\d s +F\int^t_0b(s,\mathcal{T}_s,\mathbb{P}_{(X_s,\Theta_s)},X_s(\mathcal{T}_s))\varepsilon_s^{-1}(\mathcal{T}_s)\d s\rt],
	\end{aligned}
	\end{equation*}
	i.e.,		
	\begin{equation*}
	\begin{aligned}
	&\E\lt[\int^t_0\sigma(s,\mathbb{P}_{(X_s,\Theta_s)})\varepsilon_s^{-1}(\mathcal{T}_s)\mathbb{D}^H_sF\d s\rt]\\=&\E\left[F\lt(X_t(\mathcal{T}_t)\varepsilon_t^{-1}(\mathcal{T}_t)-\xi-
	\int^t_0b(s,\mathcal{T}_s,\mathbb{P}_{(X_s,\Theta_s)},X_s(\mathcal{T}_s))\varepsilon_s^{-1}(\mathcal{T}_s)\d s\rt)\rt].
	\end{aligned}
    \end{equation*}
    Observing that $(X_t(\mathcal{T}_t)\varepsilon_t^{-1}(\mathcal{T}_t)-\xi-
    \int^t_0b(s,\mathcal{T}_s,\mathbb{P}_{(X_s,\Theta_s)},X_s(\mathcal{T}_s))\varepsilon_s^{-1}(\mathcal{T}_s)\d s$ is square integrable, we see from Definition \ref{prop_duality} that $\lt(\sigma(s,\mathbb{P}_{(X_s,\Theta_s)})\varepsilon_s^{-1}(\mathcal{T}_s)\rt)I_{[0,t]}(s), s\in[0,T]$ belongs to $\textrm{Dom}(\delta_H)$, and    
		\begin{equation*}
			\begin{aligned}
		&\int^t_0\sigma(s,\mathbb{P}_{(X_s,\Theta_s)})\varepsilon_s^{-1}(\mathcal{T}_s)\d B_s^H\\
		=&X_t(\mathcal{T}_t)\varepsilon_t^{-1}(\mathcal{T}_t)-\xi-\int^t_0 b(s,\mathcal{T}_s,\mathbb{P}_{(X_s,\Theta_s)},X_s(\mathcal{T}_s))\varepsilon_s^{-1}(\mathcal{T}_s)\d s.
		\end{aligned}
		\end{equation*}
	But this is exactly equation (\ref{eq:Gir-semilinear}). The proof that any solution of equation (\ref{eq:Gir-semilinear}) solves also \eqref{eq:semilinear} uses the same argument.
\end{proof}

Now let us focus on equation (\ref{eq:semilinear}).
 We have the following existence and uniqueness result.
\begin{theorem}
	\label{thm-gir-unique}
Equation \eqref{eq:semilinear} admits a unique solution $X=\{X_t, t\in[0,T]\}\in L^{2,*}([0,T];\R)$.
\end{theorem}
\begin{proof}
	Given a process $X^n$ such that $\sup_{t\in[0,T]}\E\lt[|X^n_t|^2\ep^{-1}_t\rt]<+\infty$, we recursively define $X^{n+1}$ as:  
	$X^0=\xi,$
	and for $n\ge0$, $X^{n+1}_t=Y^{n+1}_t(\mathcal{T}_t)\varepsilon_t^{-1}(\mathcal{T}_t)$, where
	\begin{equation*}
	Y^{n+1}_t=\xi+\int^t_0\sigma(s,\mathbb{P}_{(X^n_s,\Theta_s)})\varepsilon_s^{-1}(\mathcal{T}_s)\d B_s^H +\int^t_0 b(s,\mathcal{T}_s,\mathbb{P}_{(X^n_s,\Theta_s)},X^n_s(\mathcal{T}_s))\varepsilon_s^{-1}(\mathcal{T}_s)\d s.
	\end{equation*}
	Then from the linear growth of $b$ we  have,
	\begin{equation*}
	\begin{aligned}
		&\sup_{t\in[0,T]}\E\lt[|X^{n+1}_t|^2\ep^{-1}_t\rt] =\sup_{t\in[0,T]}\E\lt[|X^{n+1}_t(\mathcal{T}_t)\ep^{-1}_t(\mathcal{T}_t)|^2\rt]\\
		\le&2\E[\xi^2]+2\sup_{t\in[0,T]}\E\lt[\lt|\int^t_0(\sigma(s,\mathbb{P}_{(X^n_s,\Theta_s)}))\ep_s^{-1}(\T_s)\d B_s^H\rt|^2\rt]\\
		&+2\sup_{t\in[0,T]}\E\lt[\lt|\int^t_0b(s,\T_s,\PP_{(X^n_s,\Theta_s)},X^n_s(\T_s))\ep_s^{-1}(\T_s)\d s\rt|^2\rt]\\
		\le &2\E[\xi^2]+2\sup_{t\in[0,T]}\E\lt[\lt|\int^t_0(\sigma(s,\mathbb{P}_{(X^n_s,\Theta_s)}))\ep_s^{-1}(\T_s)\d B_s^H\rt|^2\rt]\\ &+2T\sup_{t\in[0,T]}\E\lt[\int^t_0\lt(1+\lt(\E\lt[|X^n_s|\rt]\rt)^2+\lt|X^n_s(\T_s)\rt|^2\rt)\ep^{-2}_s(\T_s)\d s\rt].\\
		\end{aligned}
	\end{equation*}
From the assumption that $X^n\in L^{2,*}([0,T];\R)$ and Proposition \ref{prop_malliavin} in Section 2, we get
	\begin{equation*}
		\begin{aligned}
		&\sup_{t\in[0,T]}\E\lt[|X^{n+1}_t|^2\ep^{-1}_t\rt]\\
		\le&C+2\sup_{t\in[0,T]}\E\lt[\int^t_0\lt|K_H^*(\sigma(\cdot,\mathbb{P}_{(X^n_\cdot,\Theta_\cdot)})1_{[0,t]}(\cdot)\ep_{\cdot}^{-1}(\T_{\cdot}))(s)\rt|^2\d s\rt]\\
		&+4\sup_{t\in[0,T]}
		\E\bigg[\int^t_0\int^s_0 \mathbb{D}^H_s(\sigma(r,\mathbb{P}_{(X^n_r,\Theta_r)})1_{[0,t]}(r)\ep_{r}^{-1}(\T_{r})) \\
		&\times  \mathbb{D}^H_r(\sigma(s,\mathbb{P}_{(X^n_s,\Theta_s)})1_{[0,t]}(s)\ep_{s}^{-1}(\T_{s}))\d r\d s \bigg]\\
		:=&C+2I_1+4I_2.
	\end{aligned}
	\end{equation*}
	Now for the term $I_1$, we have
		\begin{equation*}
		\begin{aligned}
		I_1=&\sup_{t\in[0,T]}\E\lt[C_H^2\int^t_0\lt|\int^t_s\sigma(r,\mathbb{P}_{(X^n_r,\Theta_r)})\ep_r^{-1}(\T_r)\lt(\frac{r}{s}\rt)^{H-\frac12}(r-s)^{H-\frac32}\d r\rt|^2\d s\rt]\\
		\le&C\E\lt[\sup_{r\in[0,T]}\ep_r^{-2}(\T_r)\rt]\sup_{t\in[0,T]}\int^t_0\lt|K_H(t,s)\rt|^2\d s\\
		\le&C T^{2H}.
		\end{aligned}
		\end{equation*} 
	For the term $I_2$, we have from relation (\ref{def_DDH}) that
	\begin{equation*}
	\begin{aligned}
	&I_2	=C_H^2\sup_{t\in[0,T]}\E\bigg[\int^t_0\int^s_0\int^r_0|s-u|^{2H-2}\sigma (r,\mathbb{P}_{(X^n_r,\Theta_r)})\ep_r^{-1}(\T_r)\gamma_u\d u \\
	&\qquad\qquad\qquad\qquad\times\int^s_0|r-v|^{2H-2}\sigma (s,\mathbb{P}_{(X^n_s,\Theta_s)})\ep_s^{-1}(\T_s) \gamma_v  \d v\d r\d s \bigg]\\
	\le& C\E\lt[\sup_{r\in[0,T]}\ep_r^{-2}(\T_r)\rt]\sup_{t\in[0,T]}\int_0^t\int^s_0\int^r_0|s-u|^{2H-2}\d u \int^s_0|r-v|^{2H-2}\d v \d r\d s\\
	\le&C\sup_{t\in[0,T]}\int_0^t\int^s_0 \lt(s^{2H-1}-(s-r)^{2H-1}\rt)\lt(r^{2H-1}+(s-r)^{2H-1}\rt)\d r\d s\\
	\le& CT^{4H}.
	\end{aligned}
	\end{equation*}
	Hence for $X^n\in L^{2,*}([0,T];\R)$, we deduce that $X^{n+1}\in L^{2,*}([0,T];\R)$.

	In the following we prove the convergence of $X^n\in L^{2,*}([0,T];\R)$. We divide the proof into 4 steps.
	
	\textbf{Step 1.} 
	Define $\overline{X}^n_t=X^n_t-X^{n-1}_t$ and $\rho^n(t)=\sigma\lt(t,\mathbb{P}_{(X^n_t,\Theta_t)}\rt)-\sigma\lt(t,\mathbb{P}_{(X^{n-1}_t,\Theta_t)}\rt)$. Notice that $\rho^n$ is a deterministic function. Then we have 
	\begin{equation}
	\label{eq:Xbar}
	\begin{aligned}
	&\E\lt[\lt|\overline{X}^{n+1}_t\rt|^2\varepsilon_t^{-1}\rt]=\E\lt[\lt|\overline{X}^{n+1}_t(\mathcal{T}_t)\varepsilon_t^{-1}(\mathcal{T}_t)\rt|^2\rt]\\
	&\le 2\E\lt[\lt|\int^t_0(\sigma(s,\mathbb{P}_{(X^n_s,\Theta_s)})-\sigma(s,\mathbb{P}_{(X^{n-1}_s,\Theta_s)}))\ep_s^{-1}(\T_s)\d B_s^H\rt|^2\rt]\\
	&+2\E\lt[\lt|\int^t_0(b(s,\T_s,\PP_{(X^n_s,\Theta_s)},X^n_s(\T_s))-b(s,\T_s,\mathbb{P}_{(X^{n-1}_s,\Theta_s)},X^{n-1}_s(\T_s)))\ep_s^{-1}(\T_s)\d s\rt|^2\rt]\\
	&=2\E\lt[\int^t_0\lt|K_H^*(\rho^n(\cdot)1_{[0,t]}(\cdot)\ep_{\cdot}^{-1}(\T_{\cdot}))(s)\rt|^2\d s\rt]\\
	&+4\E\bigg[\int^t_0\int^s_0\mathbb{D}_s^H\lt(\rho^n(r)1_{[0,t]}(r)\ep_{r}^{-1}(\T_{r})\rt) \mathbb{D}_r^H\lt(\rho^n(s)1_{[0,t]}(s)\ep_{s}^{-1}(\T_{s})\rt)\d r\d s \bigg]\\
	&+2\E\lt[\lt|\int^t_0(b(s,\T_s,\PP_{(X^n_s,\Theta_s)},X^n_s(\T_s))-b(s,\T_s,\mathbb{P}_{(X^{n-1}_s,\Theta_s)},X^{n-1}_s(\T_s)))\ep_s^{-1}(\T_s)\d s\rt|^2\rt]\\
	&=2I_3(t)+4I_4(t)+2I_5(t).
	\end{aligned}
	\end{equation}
	
	Now we deal with $I_3(t)$, $I_4(t)$ and $I_5(t)$ separately.
	
	\textbf{Step 2. The term $I_3(t)$.}  
	
	From the definition of operator $K_H^*$, we have 
	\begin{equation*}
	\begin{aligned}
	I_3(t)=&\E\lt[\int^t_0\lt|K_H^*(\rho^n(\cdot)1_{[0,t]}(\cdot)\ep_{\cdot}^{-1}(\T_{\cdot}))(s)\rt|^2\d s\rt]\\
	\le&\E\lt[\sup_{r\in[0,T]}\ep_r^{-2}(\T_r)\rt] C_H^2\int^t_0\lt(\int^t_s|\rho^n(r)|\lt(\frac{r}{s}\rt)^{H-\frac12}(r-s)^{H-\frac32}\d r\rt)^2\d s\\
	\le&C\int^t_0\lt|K_H^*(\rho^n(\cdot)1_{[0,t]}(\cdot))(s)\rt|^2\d s.
	\end{aligned}
	\end{equation*} 
	
	Let $q>1$ be adjoint to $p>\frac{1}{H-\frac12}$: $\frac 1p+\frac 1q=1$. Then, $1<q<\frac{1}{\frac32-H}<2$.   Observe that
	\begin{equation}
	\label{eq:kw}
	\begin{aligned}
	&\int^t_0\lt|K_H^*(\rho^n(\cdot)1_{[0,t]}(\cdot))(s)\rt|^2\d s\\
	\le&C\int^t_0\lt(\int_s^t|\rho^n(r)|^p\d r\rt)^{\frac 2p}\lt(\int^t_s \lt[\lt(\frac r s\rt)^{H-\frac12}(r-s)^{H-\frac32}\rt]^q\d r\rt)^{\frac 2q}\d s\\
	\le&C\lt(\int_0^t|\rho^n(r)|^p\d r\rt)^{\frac 2p}\int^t_0\lt(\frac t s\rt)^{2H-1}\lt(\int^t_s (r-s)^{q(H-\frac32)}\d r\rt)^{\frac 2q}\d s.
	\end{aligned}
	\end{equation} 
	Hence, as $q(H-\frac32)>-1$, this yields
	\begin{equation}
	\label{eq:kw1}
	\begin{aligned}
	&\int^t_0\lt|K_H^*(\rho^n(\cdot)1_{[0,t]}(\cdot))(s)\rt|^2\d s\\
	\le&C\lt(\int_0^t|\rho^n(r)|^p\d r\rt)^{\frac 2p}\int^t_0\lt(\frac t s\rt)^{2H-1}\lt(\frac{1}{q\lt(H-\frac32\rt)+1}\rt)^{\frac2q}(t-s)^{2H-3+\frac2q}\d s\\
\le &C\lt(\frac{1}{q\lt(H-\frac32\rt)+1}\rt)^{\frac2q}\frac{t^{2H-2+\frac2q}}{2-2H}\lt(\int_0^t|\rho^n(r)|^p\d r\rt)^{\frac 2p}.
	\end{aligned}
	\end{equation} 
	Since on the other hand from the Lipschitz continuity of $\sigma$ with respect to the 1-Wasserstein metric $W_1$ we have 
	\begin{equation*}
	\begin{aligned}
	&\lt(\int_0^t|\rho^n(r)|^p\d r\rt)^{\frac 2p}
	=\lt(\int_0^t|\sigma(r,\PP_{(X^n_r,\Theta_r)})-\sigma(r,\PP_{(X^{n-1}_r,\Theta_r)})|^p\d r\rt)^{\frac 2p}\\
	\le &C\lt(\int^t_0\lt(\E\lt[|\overline{X}^n_r|\rt]\rt)^p\d r\rt)^{\frac 2 p}
	\le  C\lt(\int^t_0\lt(\E\lt[|\overline{X}^n_r|^2\ep_r^{-1}\rt]\rt)^\frac p2\d r\rt)^{\frac 2 p}, 
	\end{aligned}
	\end{equation*}
	we obtain 
	\begin{equation}
	\label{eq:I_1}
	I_3(t)\le C\lt(\frac{1}{q\lt(H-\frac32\rt)+1}\rt)^{\frac2q}\frac{t^{2H-2+\frac2q}}{2-2H} \lt(\int^t_0\lt(\E\lt[|\overline{X}^n_r|^2\ep_r^{-1}\rt]\rt)^\frac p2\d r\rt)^{\frac 2 p}.
	\end{equation}
	
	\textbf{Step 3. The term $I_4(t)$.}
	
	Now we deal with the term $I_4(t)$, which can be written as 
	\begin{equation*}
	\begin{aligned}
	I_4(t)=&\E\bigg[\int^t_0\int^s_0\mathbb{D}_s^H\lt(\rho^n(r)1_{[0,t]}(r)\ep_{r}^{-1}(\T_{r})\rt) \mathbb{D}_r^H\lt(\rho^n(s)1_{[0,t]}(s)\ep_{s}^{-1}(\T_{s})\rt)\d r\d s \bigg]\\
	=&C_H^2\E\bigg[\int^t_0\int^s_0\int^r_0|s-u|^{2H-2}\rho^n(r)\ep_r^{-1}(\T_r)\gamma_u\d u \\
	&\qquad\qquad\qquad\qquad\times\int^s_0|r-v|^{2H-2}\rho^n(s)\ep_s^{-1}(\T_s) \gamma_v  \d v\d r\d s \bigg]\\
	\le& C\E\lt[\sup_{r\in[0,T]}\ep_r^{-2}(\T_r)\rt]\int_0^t\int^s_0|\rho^n(r)||\rho^n(s)|\\
	&\qquad\qquad\qquad\qquad\times\int^r_0|s-u|^{2H-2}\d u \int^s_0|r-v|^{2H-2}\d v \d r\d s.
	\end{aligned}
	\end{equation*}
	Following a similar argument to the first part of this proof and with the same $p$ as in Step 2, we have
	\begin{equation*}
	\begin{aligned}
	I_4(t)\le& C\int_0^t\int^s_0|\rho^n(r)||\rho^n(s)|\lt(s^{2H-1}-(s-r)^{2H-1}\rt)\lt(r^{2H-1}+(s-r)^{2H-1}\rt)\d r\d s\\
	\le&C t^{4H-2+\frac2q}\lt(\int^t_0|\rho^n(s)|^p\d s\rt)^{\frac{2}{p}}.
	\end{aligned}
	\end{equation*}
	From the computations in Step 2, we have
	\begin{equation}
	\label{eq:I_2}
	I_4(t)\le Ct^{4H-2+\frac2q}\lt(\int^t_0\lt(\E\lt[\lt|\overline{X}^n_r\rt|^2\ep_r^{-1}\rt]\rt)^{\frac p2}\d r\rt)^{\frac2p}.
	\end{equation}
	
	\textbf{Step 4. The term $I_5(t)$.}
	
	From the Lipschitz continuity of function $b$, for $p>2$, we have 
	\begin{equation}
	\label{eq:I_5}
	\begin{aligned}
	&I_5(t)\\
	=&\E\lt[\lt|\int^t_0(b(s,\T_s,\PP_{(X^n_s,\Theta_s)},X^n_s(\T_s))-b(s,\T_s,\mathbb{P}_{(X^{n-1}_s,\Theta_s)},X^{n-1}_s(\T_s)))\ep_s^{-1}(\T_s)\d s\rt|^2\rt]\\
	\le& t\E\lt[\int^t_0\lt|(b(s,\T_s,\PP_{(X^n_s,\Theta_s)},X^n_s(\T_s))-b(s,\T_s,\mathbb{P}_{(X^{n-1}_s,\Theta_s)},X^{n-1}_s(\T_s)))\ep_s^{-1}(\T_s)\rt|^{2}\d s\rt]\\
	\le& Ct\E\lt[\int^t_0\lt(\lt(\E\lt[\lt|\overline{X}^n_s \rt|\rt]\rt)^2 +|\overline{X}^n_s(\T_s)|^2\rt)\ep_s^{-2}(\T_s)\d s\rt]\\
	\le& Ct\E\lt[\sup_{s\in[0,T]}\ep_s^{-2}(\T_s)\rt]\int^t_0 \lt(\E\lt[\lt|\overline{X}^n_s\rt|\rt]\rt)^2\d s +Ct\E\lt[\int^t_0 |\overline{X}^n_s(\T_s)|^2\ep_s^{-2}(\T_s)\d s\rt]\\
	\le& Ct\int^t_0\E\lt[\lt|\overline{X}^n_s\rt|^2\ep_s^{-1}\rt]\d s\\
	\le &Ct^{1+\frac{p-2}{p}}\lt(\int^t_0\lt(\E\lt[\lt|\overline{X}^n_s\rt|^2\ep_s^{-1}\rt]\rt)^{\frac p 2}\d s\rt)^{\frac 2p}.
	\end{aligned}
	\end{equation}
	
	From equation (\ref{eq:Xbar}) and by combining the inequalities (\ref{eq:I_1}), (\ref{eq:I_2}) and (\ref{eq:I_5}) together, we deduce that 
	\begin{equation*}
	\begin{aligned}
	&\E\lt[|\overline{X}^{n+1}_t|^2\ep^{-1}_t\rt]\\
	\le& C_p\lt(t^{2H-2+\frac2q}+t^{4H-2+\frac2q}+t^{1+\frac{p-2}{p}}\rt)\lt(\int^t_0\lt(\E\lt[\lt|\overline{X}^n_r\rt|^2\ep^{-1}_r\rt]\rt)^{\frac p2}\d r\rt)^{\frac2p},
	\end{aligned}
	\end{equation*}
	which is equivalent to 
	\begin{equation}
	\label{picard}
	\begin{aligned}
	\lt(\E\lt[|\overline{X}^{n+1}_t|^2\ep^{-1}_t\rt]\rt)^{\frac p2} 
	\le& C_p\lt(t^{2H-2+\frac2q}+t^{4H-2+\frac2q}+t^{1+\frac{p-2}{p}}\rt)\int^t_0\lt(\E\lt[\lt|\overline{X}^n_r\rt|^2\ep^{-1}_r\rt]\rt)^{\frac p2}\d r\\
	\le & C_{p,T} \int^t_0\lt(\E\lt[\lt|\overline{X}^n_r\rt|^2\ep^{-1}_r\rt]\rt)^{\frac p2}\d r.
	\end{aligned}
	\end{equation}
	
	By the Picard iteration, we get 
	$$
	\lt(\E\lt[|\overline{X}^{n+1}_t|^2\ep^{-1}_t\rt]\rt)^{\frac p2} 
	\le C_0C^n \frac{t^n}{n!}.
	$$
	Hence 
	$$
	\sup_{t\in[0,T]}\E\lt[|\overline{X}^{n+1}_t|^2\ep^{-1}_t\rt]
	\le C_0C^\frac{2n}{p} \lt(\frac{T^n}{n!}\rt)^{\frac2p}.
	$$
	This means $X^n$ is a Cauchy sequence in $L^{2,*}([0,T];\R)$ and the limit $X$ is a solution of equation \eqref{eq:Gir-semilinear}, and thus also of (\ref{eq:semilinear}) (See Theorem \ref{thm_semi-girsemi}). 
	
	\smallskip
	 Let us now show that the solution of (24) is unique. For this we consider two solutions $X$ and $Y$ of equation (24). Repeating the argument developed in the frame of the Picard iteration, we get in analogy to (31)
	 $$(\E[|X_t-Y_t|^2\varepsilon_t^{-1}])^{p/2}\le C_{p,T}\int_0^t\lt(\E[|X_r-Y_r|^2\varepsilon_r^{-1}]\rt)^{p/2}dr,\ t\in[0,T],$$
	 and, thus, Gronwall's inequality yields that $X_t=Y_t,\ \mathbb{P}$-a.s., $t\in[0,T]$. Therefore, the uniqueness holds true and the proof is complete.
    \end{proof}
\begin{remark}\label{remark}
	One can see in the Step 4 of the above proof, assumption (H1) is essential. However, if we consider the equation with $\gamma_s\equiv0$, i.e., $\ep_s\equiv 1$ in equation \eqref{eq:Gir-semilinear} (which is the case that we consider in the next section), then we only need the following assumption on $\sigma$ and $b$:
	
	\textbf{(H1$^\pr$)} For any $s\in[0,T]$, $x, x^\prime\in\R$, $\eta, \eta^\prime\in L^2(\Omega,\mathcal{F}, \PP;\R)$ and $\Theta\in L^2(\Omega,\mathcal{F}, \PP;$ $\R^m)$, there exists a constant $C>0$ such that
	$$
	|\sigma(s,\PP_{(\eta,\Theta)})|\le C\lt(1+\lt(\E\lt[|\eta^2|\rt]\rt)^\frac12\rt),
	$$ 
	$$
    |b(s,\PP_{(\eta,\Theta)},x)|\le C\lt(1+\lt(\E\lt[|\eta^2|\rt]\rt)^\frac12
    +|x|\rt),
	$$
	$$ 
	\lt|\sigma(s, \PP_{(\eta,\Theta)})-\sigma(s,\PP_{(\eta^\prime,\Theta)})\rt|\le C W_2(\eta,\eta^\pr), 
	$$
	$$
	|b(s,\PP_{(\eta,\Theta)},x)-b(s,\PP_{(\eta^\prime,\Theta)},x^\prime)|\le C \lt(W_2(\eta,\eta^\pr)+|x-x^\prime|\rt).
	$$
	Moreover, if $\gamma_s\equiv0$, the space $L^{2,*}([0,T];\R)$ becomes the classical space $L^2_{\mathbb{F}}([0,T];\R)$, the space of $\mathbb{F}$-adapted square integrable processes.
\end{remark}

\section{Mean-field stochastic control problem driven by fractional Brownian motion with $H>1/2$}
\label{section control}
In this section, we study a mean-field stochastic control problem driven by a fractional Brownian motion $B^H$ with $H>1/2$. 

Let $U$ be a nonempty bounded convex subset of $\R^m$. We define the space of admissible controls as follows:
$$
\mathcal{U}([0,T]):=\lt\{u:[0,T]\times \Omega\to U\big|u \textrm{ is an} \ \mathbb{F}\textrm{-adapted process}\rt\}.
$$ 
We consider the following dynamics for our mean-field controlled system:
\begin{equation}
\label{eq:control}
X_t^u=x+\int^t_0\sigma(\PP_{X_s^u})\d B^H_s+\int^t_0 b(\PP_{(X_s^u,u_s)},X_s^u, u_s)\d s,
\end{equation}
where $x\in\R$, and $u\in\mathcal{U}([0,T])$ is an admissible control process. For any given  $u\in\mathcal{U}([0,T])$, we know from Theorem \ref{thm_semi-girsemi} that there exists a unique solution to the controlled system (\ref{eq:control}). In fact, (\ref{eq:control}) constitute a particular case of equation (\ref{eq:semilinear}) with $\gamma\equiv 0$ (See Remark \ref{remark}). 

The cost functional is assumed to be depend on a running cost function $f:[0,T]\times \mathcal{P}_2(\R\times U)\times\R\times U\to \R$ and a terminal cost function $g: \R\times \mathcal{P}_2(\R)\to \R$:
\begin{equation}
\label{eq:cost}
J(u)=\E\lt[\int^T_0f\lt(\PP_{(X^u_t,u_t)},X^u_t,u_t\rt)\d t +g\lt(X_T^u,\PP_{X^u_T}\rt)\rt]. 
\end{equation}
Our aim is to characterise an optimal control $u^*\in \mathcal{U}([0,T])$ such that
\begin{equation}
\label{eq:optimal cost}
J(u^*)=\inf_{u\in \mathcal{U}([0,T])}J(u).
\end{equation}
If there exists such optimal control $u^*$, we call the corresponding pair $(X^*,u^*)$  optimal  for  the control problem. Here $X^*=X^{u^*}$ denotes the solution of \eqref{eq:control} associated with the control process $u^*$.

The main purpose of this section is to find a necessary condition under which the pair  $(X^*,u^*)$ is optimal. This condition will be based on Pontryagin's maximum principle. 

To achieve this goal, we make first the following assumptions on the coefficients $\sigma: \mathcal{P}_2(\R)\to \R, b: \mathcal{P}_2(\R\times U)\times\R\times U\to \R$,  $g:\R\times\mathcal{P}_2(\R)\to \R$ and $f: \mathcal{P}_2(\R\times U)\times\R\times U\to \R$. 

(\textbf{H2}) $\sigma, b, g,f$ are Lipschitz continuous, i.e., there exists a constant $C>0$ such that
\begin{itemize}
	
	\item[(i)] $\lt|\sigma(\mu)-\sigma(\mu^\pr)\rt|\le C W_2(\mu,\mu^\pr)$, for any $\mu, \mu^\prime\in \mathcal{P}_2(\R)$;
	\item[(ii)] $|b(\mu,x,u)-b(\mu^\pr,x^\prime,u^\pr)|\le C \lt(W_2(\mu,\mu^\pr)+|x-x^\prime|+|u-u^\pr|\rt)$, for any $(\mu,x,u)$, $(\mu^\pr,x^\pr,u^\pr)\in\mathcal{P}_2(\R\times U)\times \R\times U$;
	\item[(iii)] $|g(x,\mu)-g(x^\pr,\mu^\pr)|\le C(|x-x^\pr|+W_2(\mu,\mu^\pr)),$ for any $(x,\mu), (x^\pr,\mu^\pr)\in\R\times\mathcal{P}_2(\R)$,
	\item[(ii)] $|f(\mu,x,u)-f(\mu^\pr,x^\prime,u^\pr)|\le C \lt(W_2(\mu,\mu^\pr)+|x-x^\prime|+|u-u^\pr|\rt)$, for any $(\mu,x,$ $u)$, $(\mu^\pr,x^\pr,u^\pr)\in\mathcal{P}_2(\R\times U)\times \R\times U$;
\end{itemize}

(\textbf{H3})  $\sigma$ is differentiable in $(\mu,x)\in\mathcal{P}_2(\R)\times\R$, and the derivative $\partial_{\mu} \sigma: \mathcal{P}_2(\R)\times\R\to\R$ is bounded and Lipschitz continuous, i.e., there exists a constant $C>0$ such that
\begin{itemize}
	\item[(i)] $|\partial_\mu \sigma(\mu,y)|\le C$, for any $(\mu,y)\in\mathcal{P}_2(\R)\times\R$;
	\item[(ii)] $|\partial_\mu \sigma(\mu,y)-\partial_\mu\sigma(\mu^\pr,y^\pr)|\le C(W_2(\mu,\mu^\pr)+|y-y^\pr|),$ for any $(\mu,y)$, $(\mu^\pr,y^\pr)\in\mathcal{P}_2(\R)\times\R$.
\end{itemize} 

(\textbf{H4}) For $w=b$ and $f$, $w$ is differentiable in $(\mu,x,u)\in\mathcal{P}_2(\R\times U)\times\R\times U$ and the derivatives $\partial_\mu w:\mathcal{P}_2(\R\times U)\times \R\times U\times(\R\times U)\to \R\times U$, $\partial_x w:\mathcal{P}_2(\R\times U)\times \R\times U\to\R$ and $\partial_u w:\mathcal{P}_2(\R\times U)\times \R\times U\to U$ are bounded and Lipschitz continuous, i.e., there exists a constant $C>0$ such that
\begin{itemize}
 	\item[(i)] $|\partial_\mu w(\mu,x,u,y)|+|\partial_x w(\mu,x,u)|+|\partial_u w(\mu,x,u)|\le C$, for any $(\mu,x,u,y)\in\mathcal{P}_2(\R\times U)\times \R\times U\times (\R\times U)$;
 	\item[(ii)] $|\partial_\mu w(\mu,x,u,y)-\partial_\mu w(\mu^\pr,x^\pr,u^\pr,y^\pr)|\le C(W_2(\mu,\mu^\pr)+|x-x^\pr|+|u-u^\pr|+|y-y^\pr|),$ for any $(\mu,x,u,y), (\mu^\pr,x^\pr,u^\pr,y^\pr)\in\mathcal{P}_2(\R\times U)\times \R\times U\times(\R\times U)$;
 	\item[(iii)] $|\partial_x w(\mu,x,u)-\partial_x w(\mu^\pr,x^\pr,u^\pr)|\le C(W_2(\mu,\mu^\pr)+|x-x^\pr|+|u-u^\pr|),$ for any $(\mu,x,u), (\mu^\pr,x^\pr,u^\pr)\in\mathcal{P}_2(\R\times U)\times \R\times U$;
 	\item[(iv)] $|\partial_u w(\mu,x,u)-\partial_u w(\mu^\pr,x^\pr,u^\pr)|\le C(W_2(\mu,\mu^\pr)+|x-x^\pr|+|u-u^\pr|),$ for any $(\mu,x,u), (\mu^\pr,x^\pr,u^\pr)\in\mathcal{P}_2(\R\times U)\times \R\times U$.
\end{itemize} 
 
 \textbf{(H5)} $g$ is differentiable in $(x,\mu)\in\R\times\mathcal{P}_2(\R)$ and the derivatives $\partial_x g: \R\times\mathcal{P}_2(\R)\to\R$ and $\partial_\mu g:\R\times\mathcal{P}_2(\R)\times\R\to \R$ are bounded.

For any $\ep\in [0,1]$ and $u\in\mathcal{U}([0,T])$, let $u^\ep=u^*+\ep(u-u^*)$. Observe that, thanks to the convexity of $U$, $u^\ep\in\mathcal{U}([0,T])$. We denote by $X^\ep$ the solution of equation (\ref{eq:control}) with $u$ replaced by $u^\ep$.  
\begin{lemma}
\label{lemma_Y}
The following SDE obtained by formal differentiation of \eqref{eq:control} for $u=u^\ep$ with respect to $\ep$ at $\ep=0$, for $t\in[0,T]$,
\begin{equation}
\label{eq:Y}
\begin{aligned}
Y_t=&\int^t_0\wt{\E}\lt[\partial_{\mu}\sigma\lt(\PP_{X^*_s}, \wt{X}_s^*\rt)\wt{Y}_s\rt]\d B_s^H\\
&+\int^t_0\wt{\E}\lt[\lt\langle\partial_{\mu}b\lt(\PP_{(X_s^*,u^*_s)}, X^*_s, u^*_s, \wt{X}_s^*,\wt{u}^*_s\rt),\lt(\wt{Y}_s, \wt{u}_s-\wt{u}^*_s\rt)\rt\rangle\rt]\d s\\
&+\int^t_0\partial_xb\lt(\PP_{(X^*_s,u^*_s)}, X_s^*, u^*_s\rt)Y_s\d s+\int^t_0\partial_ub\lt(\PP_{(X^*_s,u^*_s)}, X_s^*, u^*_s\rt)(u_s-u_s^*)\d s, 
\end{aligned}
\end{equation}
has  a unique solution $Y=(Y_t)_{t\in[0,T]}\in L^{2}([0,T];\R)$. Moreover, 
\begin{equation}
\label{eq:L2-Y}
\lim_{\ep\downarrow0} \sup_{t\in[0,T]}\E\lt[\lt|Y_t-\frac{X^\ep_t-X^*_t}{\ep}\rt|^2\rt]=0.
\end{equation}
In the above equation, $(\wt{X}^*,\wt{Y},\wt{u},\wt{u}^*)$ is an independent copy of $(X^*,Y,u,u^*)$ defined on a probability space $(\wt{\Omega},\wt{\F}, \wt{\PP})$. The expectation $\wt{\E}[\cdot]$ under $\wt{\PP}$ only concerns $(\wt{X}^*,\wt{Y},\wt{u},\wt{u}^*)$ but not $(X^*,Y,u,u^*)$.
\end{lemma}

\begin{remark}
	With the above convention concerning $\wt{\E}[\cdot]$ we have, 
	$$
	\wt{\E}\lt[\partial_\mu\sigma (\PP_{X^*_s},\wt{X}^*_s)\wt{Y}^*_s\rt]=\E\lt[\partial_\mu \sigma (\PP_{X^*_s},X^*_s)Y_s\rt],
	$$ 
	and
	\begin{equation*}
		\begin{aligned}
		&\wt{\E}\lt[\langle\partial_\mu b(\PP_{(X^*_s,u^*_s)},X^*_s,u^*_s,\wt{X}^*_s,\wt{u}^*_s),(\wt{Y}_s,\wt{u}_s-\wt{u}^*_s)\rangle\rt]\\
		=&\E\lt[\langle\partial_\mu b(\PP_{(X^*_s,u^*_s)},x,v,{X}^*_s,{u}^*_s),({Y}_s,{u}_s-{u}^*_s)\rangle\rt]\big|_{x=X^*_s,v=u^*_s}.
		\end{aligned}
	\end{equation*}

\end{remark}
\begin{proof}(of Lemma \ref{lemma_Y}). The existence and uniqueness of the solution $Y\in L^2([0,$ $T];\R)$ is a special case of \eqref{eq:semilinear}. Indeed, for $\Theta_s=(X^*_s,u^*_s,u_s)$, $s\in[0,T]$, $\eta\in L^2(\Omega,\mathcal{F},\PP)$, $\omega\in\Omega$, we can choose the coefficients in equation \eqref{eq:semilinear} as follows:
\begin{equation*}
\begin{aligned}
 \gamma_s:=& 0;\\ \bar{\sigma}(s,\PP_{(\eta,\Theta_s)}):=&\wt{\E}\lt[\partial_\mu \sigma (\PP_{X^*_s},\wt{X}^*_s)\wt{\eta}\rt];\\
\bar{b}(s,\PP_{(\eta_s,\Theta_s)},\eta,\omega):=&\wt{\E}\lt[\lt\langle\partial_{\mu}b\lt(\PP_{(X_s^*,u^*_s)}, X^*_s(\omega), u^*_s(\omega), \wt{X}_s^*,\wt{u}^*_s\rt),\lt(\wt{\eta}, \wt{u}_s-\wt{u}^*_s\rt)\rt\rangle\rt]\\
&+\partial_xb\lt(\PP_{(X^*_s,u^*_s)}, X_s^*(\omega), u^*_s(\omega)\rt)\eta\\
&+\partial_ub\lt(\PP_{(X^*_s,u^*_s)}, X_s^*(\omega), u^*_s(\omega)\rt)(u_s(\omega)-u_s^*(\omega)),
\end{aligned}
\end{equation*}
for which we have 
\begin{equation*}
\begin{aligned}
	\bar{\sigma}(s,\PP_{(\eta,\Theta_s)})\le&C(1+\E[|\eta|]);\\
|\bar{b}(s,\PP_{(\eta,\Theta_s)},x,\omega)|\leq& C(1+\E[|\eta|]+|x|);\\
\lt|\bar{\sigma}(s,\PP_{(\eta,\Theta_s)})-\wt{\sigma}(s,\PP_{(\eta^\pr,\Theta_s)})\rt|\le &C\wt{\E}\lt[|\eta-\eta^\pr|\rt];\\
\lt|\wt{b}(s,\PP_{(\eta,\Theta_s)},x,\omega)-\wt{b}(s,\PP_{(\eta^\pr,\Theta_s)},x^\pr,\omega)\rt|\le &C\lt(\wt{\E}\lt[|\eta-\eta^\pr|\rt]+|x-x^\pr|\rt),
\end{aligned}
\end{equation*}
for all $s\in[0,T]$, $\omega\in\Omega$, $\eta,\eta^\pr\in L^2(\Omega,\F,\PP)$ and $x, x^\pr\in\R$.
Then the result follows.

	The proof of (\ref{eq:L2-Y}) is split into 5 steps.
	
\textbf{Step 1.}
Following the same method in the proof of Theorem \ref{thm-gir-unique}, the only difference in the argument consists in Step 4 of the proof, where we have to take into account that we have now different $\Theta_s$'s. Recall also that, as $\gamma=0$ here, $\ep_t=1, t\in[0,T]$. Thus, we obtain 
\begin{equation*}
\begin{aligned}
\lt(\E\lt[\lt|X_t^*-X_t^\ep\rt|^2\rt]\rt)^{\frac p2}\le &C \int^t_0\lt(\E\lt[\lt|X^*_s-X^\ep_s\rt|^2+|u_s^*-u_s^\ep|^2\rt]\rt)^{\frac p2}\d s\\
\le &C\int^t_0\lt(\E\lt[\lt|X^*_s-X_s^\ep\rt|^2\rt]\rt)^{\frac p2}\d s+C\int^t_0\lt(\E\lt[\ep^2\lt|u^*_s-u_s\rt|^2\rt]\rt)^{\frac p2}\d s\\
\le &C\int^t_0\lt(\E\lt[\lt|X^*_s-X_s^\ep\rt|^2\rt]\rt)^{\frac p2}\d s+C\ep^p.
\end{aligned}
\end{equation*}
From Gronwall's inequality, we get
\begin{equation*}
\begin{aligned}
\lt(\E\lt[\lt|X_t^*-X_t^\ep\rt|^2\rt]\rt)^{\frac p2}\le C\ep^p.
\end{aligned}
\end{equation*}
Hence, we deduce that
\begin{equation}
\label{eq:Xstar-ep}
\begin{aligned}
\E\lt[\lt|X_t^*-X_t^\ep\rt|^2\rt]\le C\ep^2, \, t\in[0,T].
\end{aligned}
\end{equation}
This yields that, as $\ep\to 0$, $X^\ep_t$ converges to $X^*_t$ in $L^2$, whence $X^\ep_t$  also converges to $X^*_t$ in probability.
  
	\textbf{Step 2.}
Let $Y_t$ be the solution of equation (\ref{eq:Y}), we want to prove that 
$$
\lim_{\ep\downarrow0}\sup_{t\in[0,T]}\E\lt[\lt|\frac{X^\ep_t-X^*_t}{\ep}-Y_t\rt|^2\rt]=0.
$$
Indeed, we have
\begin{equation}
\label{eq:ep-Y}
\begin{aligned}
&\frac{X^\ep_t-X^*_t}{\ep}-Y_t
=\int^t_0\lt\{\frac{1}{\ep}\lt[\sigma(\PP_{X^\ep_s})-\sigma(\PP_{X_s^*})\rt]-\wt{\E}\lt[\partial_{\mu}\sigma\lt(\PP_{X_s^\ep},\wt{X}_s^*\rt) \wt{Y_s}\rt]\rt\}\d B^H_s\\
&\qquad\qquad+\int^t_0\bigg\{\frac{1}{\ep}\lt[b\lt(\PP_{(X^\ep_s,u^\ep_s)},X^\ep_s,u^\ep\rt)-b(\PP_{(X^*_s,u^*_s)}, X^*_s,u^*_s)\rt]\\
&\quad\qquad\qquad-\wt{\E}\lt[\lt\langle\partial_{\mu}b\lt(\PP_{(X_s^*,u^*_s)}, X^*_s, u^*_s, \wt{X}_s^*,\wt{u}^*_s\rt),\lt(\wt{Y}_s, \wt{u}_s-\wt{u}^*_s\rt)\rt\rangle\rt]\\
&\quad \qquad\qquad-\partial_xb\lt(\PP_{(X^*
	_s,u^*_s)}, X_s^*, u^*_s\rt)Y_s-\partial_ub\lt(\PP_{(X^*_s,u^*_s)}, X_s^*, u^*_s\rt)(u_s-u_s^*)\bigg\}\d s.
\end{aligned}
\end{equation}

In what follows, for simplicity of notations and for $\theta\in[0,1]$, we denote $X_s^*+\theta(X^\ep_s-X^*_s)$ by $Z^\theta_s$, and $u_s^*+\theta(u^\ep_s-u^*_s)$ by $v^\theta_s$. Then by applying the chain rule for Fr\'echet derivatives and the definition of derivative with respect to $\mu$ (Section 2.5), we can write
\begin{equation}
\label{eq:ep-sigma}
\begin{aligned}
&\frac{1}{\ep}\lt[\sigma(\PP_{X^\ep_s})-\sigma(\PP_{X^*_s})\rt]\\
=&\frac{1}{\ep}\int^1_0\partial_{\theta}\lt[\sigma\lt(\PP_{X_s^*+\theta(X^\ep_s-X^*_s)}\rt)\rt]\d \theta
=\frac{1}{\ep}\int^1_0\partial_{\theta}\lt[\wt{\sigma}(X_s^*+\theta(X^\ep_s-X^*_s))\rt]\d \theta\\
=&\frac{1}{\ep}\int^1_0\lt(D\wt{\sigma}\rt)(X_s^*+\theta(X^\ep_s-X^*_s)(\wt{X}^\ep_s-\wt{X}^*_s)\d \theta
=\int^1_0\wt{\E}\lt[\partial_{\mu}\sigma\lt(\PP_{Z_s^\theta},\wt{Z}_s^\theta\rt)\frac{\wt{X}^\ep_s-\wt{X}^*_s}{\ep}\rt]\d \theta,
\end{aligned}
\end{equation}
and
\begin{equation}
\label{eq:ep-b}
\begin{aligned}
&\frac{1}{\ep}\lt[b(\PP_{(X^\ep_s,u^\ep_s)},X^\ep_s,u^\ep_s)-b(\PP_{(X^*_s,u^*_s)},X^*_s,u^*_s)\rt]\\
=&\frac{1}{\ep}\int^1_0\partial_{\theta}\lt[b\lt(\PP_{(Z_s^\theta,v_s^\theta)},Z_s^\theta,v_s^\theta\rt)\rt]\d \theta\\
=&\int^1_0\Bigg\{\wt{\E}\bigg[\bigg\langle\partial_{\mu}b\big(\PP_{(Z_s^\theta,v_s^\theta)},Z_s^\theta,v_s^\theta,\wt{Z}_s^\theta,\wt{v}_s^\theta\big),\lt(\frac{\wt{X}^\ep_s-\wt{X}^*_s}{\ep}, \wt{u}_s-\wt{u}^*_s\rt)\bigg\rangle\bigg]\\
&+\partial_xb\big(\PP_{(Z_s^\theta,v_s^\theta)},Z_s^\theta,v_s^\theta\big)\frac{\wt{X}^\ep_s-\wt{X}^*_s}{\ep}+\partial_ub\big(\PP_{(Z_s^\theta,v_s^\theta)}Z_s^\theta,v_s^\theta\big)(u_s-u^*_s)
\Bigg\}\d \theta.
\end{aligned}
\end{equation}
Substituting equations (\ref{eq:ep-sigma}) and (\ref{eq:ep-b}) into equation (\ref{eq:ep-Y}), we get
\begin{equation}
\label{eq:Yep-Y}
\begin{aligned}
&\frac{X^\ep_t-X^*_t}{\ep}-Y_t\\
=&\int^t_0\int^1_0\wt{\E}\lt[\partial_{\mu}\sigma\lt(\PP_{Z_s^\theta},\wt{Z}_s^\theta\rt)\frac{\wt{X}^\ep_s-\wt{X}^*_s}{\ep}-\partial_{\mu}\sigma\lt(\PP_{X_s^*},\wt{X}_s^*\rt) \wt{Y_s}\rt]\d \theta\d B^H_s\\
&+\int^t_0\int^1_0\Bigg\{\wt{\E}\bigg[\bigg\langle\partial_{\mu}b\big(\PP_{(Z_s^\theta,v_s^\theta)},Z_s^\theta,v_s^\theta,\wt{Z}_s^\theta,\wt{v}_s^\theta\big),\lt(\frac{\wt{X}^\ep_s-\wt{X}^*_s}{\ep}, \wt{u}_s-\wt{u}^*_s\rt)\bigg\rangle\bigg]\\
&\qquad+\partial_xb\big(\PP_{(Z_s^\theta,v_s^\theta)},Z_s^\theta,v_s^\theta\big)\frac{\wt{X}^\ep_s-\wt{X}^*_s}{\ep}+\partial_ub\big(\PP_{(Z_s^\theta,v_s^\theta)},Z_s^\theta,v_s^\theta\big)(u_s-u^*_s)\\
&\qquad-\wt{\E}\lt[\lt\langle\partial_{\mu}b\lt(\PP_{(X_s^*,u^*_s)}, X^*_s, u^*_s, \wt{X}_s^*,\wt{u}^*_s\rt),\lt(\wt{Y}_s, \wt{u}_s-\wt{u}^*_s\rt)\rt\rangle\rt]\\
&\qquad-\partial_xb\lt(\PP_{(X^*
	_s,u^*_s)}, X_s^*, u^*_s\rt)Y_s-\partial_ub\lt(\PP_{(X^*_s,u^*_s)}, X_s^*, u^*_s\rt)(u_s-u_s^*)\Bigg\}\d \theta\d s\\
=&I_6(t)+I_7(t).
\end{aligned}
\end{equation}
\textbf{Step 3. The term $I_6(t)$.} 
Recall that
\begin{equation*}
\begin{aligned}
I_6(t)=&\int^t_0\int^1_0\wt{\E}\lt[\lt(\partial_{\mu}\sigma\lt(\PP_{Z_s^\theta},\wt{Z}_s^\theta\rt)-\partial_{\mu}\sigma\lt(\PP_{X_s^*},\wt{X}_s^*\rt)\rt) \frac{\wt{X}^\ep_s-\wt{X}^*_s}{\ep}\rt]\d \theta\d B^H_s\\
&+\int^t_0\int^1_0\wt{\E}\lt[\partial_{\mu}\sigma\lt(\PP_{X_s^*},\wt{X}_s^*\rt)\lt(\frac{\wt{X}^\ep_s-\wt{X}^*_s}{\ep}-\wt{Y}_s\rt)\rt]\d \theta\d B^H_s.
\end{aligned}
\end{equation*}
By applying the same method and the same $p$ as in the proof of Theorem \ref{thm-gir-unique} (with $\gamma\equiv0$), observing that the integrands with respect to $B^H$ are deterministic, we get 
\begin{equation}
\label{eq:I4}
\begin{aligned}
\E\lt[I^2_6(t)\rt]\le& C\Bigg\{\int^t_0\Bigg|\E\bigg[\int^1_0\lt(\partial_{\mu}\sigma\lt(\PP_{Z_s^\theta},{X}_s^\theta\rt)-\partial_{\mu}\sigma\lt(\PP_{X_s^*},{X}_s^*\rt)\rt) \frac{{X}^\ep_s-{X}^*_s}{\ep}\d \theta\bigg]\Bigg|^p\d s\Bigg\}^{\frac2p}\\
&+C\bigg\{\int^t_0\Bigg|\E\bigg[\int^1_0\partial_{\mu}\sigma\lt(\PP_{X_s^*},{X}_s^*\rt) \bigg(\frac{{X}^\ep_s-{X}^*_s}{\ep}-Y_s\bigg)\d \theta\bigg]\Bigg|^p\d s\Bigg\}^{\frac2p}\\
\le& C\Bigg\{\int^t_0\lt(\E\lt[\lt(\lt(\E\lt[|X^\ep_s-X_s^*|^2\rt]\rt)^{\frac12}+|X^\ep_s-X_s^*|\rt)\lt|\frac{X^\ep_s-X_s^*}{\ep}\rt|\rt]\rt)^p\d s\Bigg\}^\frac{2}{p}\\
&+C\lt(\int^t_0\lt(\E\lt[\lt|\frac{X^\ep_s-X_s^*}{\ep}-Y_s\rt|\rt]\rt)^p\d s\rt)^{\frac{2}{p}}.
\end{aligned}
\end{equation}
Hence, from the results of \textbf{Step 1}, we deduce from the bounded convergence theorem that the first term of right hand side in the above inequality  converges to $0$ as $\ep\to 0$.

\textbf{Step 4. The term $I_7(t)$.} We can write $I_7(t)$ as follows:
\begin{equation*}
\begin{aligned}
I_7(t)=&\int^t_0\int^1_0\Bigg\{\wt{\E}\bigg[\bigg\langle\partial_{\mu}b\big(\PP_{(Z_s^\theta,v_s^\theta)},Z_s^\theta,v_s^\theta,\wt{Z}_s^\theta,\wt{v}_s^\theta\big)-\partial_{\mu}b\lt(\PP_{(X_s^*,u^*_s)}, X^*_s, u^*_s, \wt{X}_s^*,\wt{u}^*_s\rt),\\
&\qquad\qquad\qquad\lt(\frac{\wt{X}^\ep_s-\wt{X}^*_s}{\ep}, \wt{u}_s-\wt{u}^*_s\rt)\bigg\rangle\bigg]\\
&+\wt{\E}\lt[\lt\langle\partial_{\mu}b\lt(\PP_{(X_s^*,u^*_s)}, X^*_s, u^*_s, \wt{X}_s^*,\wt{u}^*_s\rt),\lt(\frac{\wt{X}^\ep_s-\wt{X}^*_s}{\ep}-\wt{Y}_s,0\rt)\rt\rangle\rt]\\
&+\lt(\partial_xb\lt(\PP_{(Z_s^\theta,v_s^\theta)},Z_s^\theta,v_s^\theta\rt)-\partial_xb\lt(\PP_{(X^*_s,u^*_s)}, X_s^*, u^*_s\rt)\rt)\frac{{X}^\ep_s-{X}^*_s}{\ep}\\
&+\partial_xb\lt(\PP_{(X^*
	_s,u^*_s)}, X_s^*, u^*_s\rt)\lt(\frac{{X}^\ep_s-{X}^*_s}{\ep}-{Y}_s\rt)\\
&+\lt(\partial_ub\big(\PP_{(Z_s^\theta,v_s^\theta)}Z_s^\theta,v_s^\theta\big)-\partial_ub\lt(\PP_{(X^*_s,u^*_s)}, X_s^*, u^*_s\rt)\rt)(u_s-u^*_s)\Bigg\}\d \theta\d s.
\end{aligned}
\end{equation*}
From our assumptions, we have
\begin{equation}
\label{eq:I5}
\begin{aligned}
&\E[I_7^2(t)]\\&\le C\Bigg\{\int^t_0\bigg(\E\Bigg[\lt(\lt(\E\lt[|X^\ep_s-X_s^*|^2\rt]\rt)^{\frac12}+|X^\ep_s-X_s^*|+\lt(\E[|u^\ep_s-u^*_s|^2]\rt)^{\frac12
	}+|u^\ep_s-u^*_s|\rt)\\
&\times\bigg(\lt|\frac{X^\ep_s-X_s^*}{\ep}\rt|+|u_s-u^*_s|\bigg)\Bigg]\bigg)^p\d s\Bigg\}^\frac{2}{p}+C\lt(\int^t_0\lt(\E\lt[\lt|\frac{X^\ep_s-X_s^*}{\ep}-Y_s\rt|\rt]\rt)^p\d s\rt)^{\frac{2}{p}}.
\end{aligned}
\end{equation} 
From \textbf{Step 1}, we get by the bounded convergence theorem that the first term of (\ref{eq:I5}) converges to zero when $\ep\to 0$. 

\textbf{Step 5.} 
Combining the relations (\ref{eq:Yep-Y}), (\ref{eq:I4}) and (\ref{eq:I5}), we deduce that
\begin{equation*}
\E\lt[\lt|\frac{X^\ep_t-X^*_t}{\ep}-Y_t\rt|^2\rt]\le C_\ep +C\lt(\int^t_0\lt(\E\lt[\lt|\frac{X^\ep_s-X_s^*}{\ep}-Y_s\rt|\rt]\rt)^p\d s\rt)^{\frac{2}{p}},\ t\in[0,T],
\end{equation*}
where $C_\ep\to 0$ as $\ep\to 0$.  Consequently, for $t\in[0,T],$
\begin{equation*}
\sup_{r\in[0,t]}\lt(\E\lt[\lt|\frac{X^\ep_r-X^*_r}{\ep}-Y_r\rt|^2\rt]\rt)^{\frac p2}\le C_\ep +C\int^t_0\lt(\E\lt[\lt|\frac{X^\ep_s-X_s^*}{\ep}-Y_s\rt|^2\rt]\rt)^{\frac p2}\d s, 
\end{equation*}
and Gronwall's inequality allows to conclude that
\begin{equation*}
\sup_{t\in[0,T]}\lt(\E\lt[\lt|\frac{X^\ep_t-X^*_t}{\ep}-Y_t\rt|^2\rt]\rt)^{\frac p2}\le C_\ep\to 0,\ \textrm{as}\ \ep\to 0.
\end{equation*}
This proves that $\frac{X^\ep_t-X^*_t}{\ep}$ converges to $Y_t$ in $L^2$, uniformly in $t\in[0,T]$. Our proof is completed.
\end{proof}

As $(X^*,u^*)$ is an optimal pair, $J(u^\ep)\ge J(u^*), \ep\in[0,1]$, and therefore, 
\begin{equation}
\label{eq:dJ}
\frac{\d}{\d \ep}J(u^\ep)\bigg|_{\ep=0}
:=\lim_{0<\ep\downarrow 0}\frac 1\ep(J(u^\ep)-J(u^*))\ge 0.
\end{equation}
 That is, due to Lemma \ref{lemma_Y} and the computations on $f$ and $g$,
\begin{equation}
\label{eq:dJ1}
\begin{aligned}
&0\le \frac{\d}{\d \ep}J(u^\ep)\bigg|_{\ep=0}=\E\lt[\partial_x g(X^*_T,\PP_{X^*_T})Y_T\rt]+\E\lt[\wt{\E}\lt[\partial_\mu g(X^*_T,\PP_{X^*_T},\wt{X}^*_T)\wt{Y}_T\rt]\rt]\\
&+\E\lt[\int^T_0\partial_x f(\PP_{(X^*_t,u^*_t)},X^*_t,u^*_t)Y_t\d t\rt]\\
&+\E\lt[\int^T_0\wt{\E}\lt[(\partial_\mu f)_1( \PP_{(X^*_t,u^*_t)},X^*_t,u^*_t,\wt{X}^*_t,\wt{u}^*_t)\wt{Y}_t\rt]\d t\rt]\\
&+\E\lt[\int^T_0\wt{\E}\lt[(\partial_\mu f)_2( \PP_{(X^*_t,u^*_t)},X^*_t,u^*_t,\wt{X}^*_t,\wt{u}^*_t)(\wt{u}_t-\wt{u}^*_t)\rt]\d t\rt]\\
&+\E\lt[\int^T_0\partial_u f( \PP_{(X^*_t,u^*_t)},X^*_t,u^*_t)(u_t-{u}^*_t)\d t\rt].
\end{aligned}
\end{equation}
We define an adjoint process $P_t$ by
\begin{equation}
\label{eq:p}
P_t=P_T-\int^T_t\alpha_s\d s-\int^T_t \beta_s\d W_s,
\end{equation}
where  $\{\alpha_s, s\in[0,T]$ is a square integrable progressively measurable process, which will be specified later,  $P_T=\partial_x g(X^*_T,\PP_{X^*_T})+\wt{\E}\lt[\partial_\mu g(\wt{X^*_T},\PP_{X^*_T},X^*_T)\rt]$, and $(P=(P_t)_{t\in[0,T]},$  $\beta=(\beta_t)_{t\in[0,T]})$ is the solution of the backward equation (\ref{eq:p}). In particular,  $P=(P_t)_{t\in[0,T]}$ and $\beta=(\beta_t)_{t\in[0,T]}$ are square integrable, progressively measurable processes.

We apply Ito's formula (Corollary \ref{corol}) to $Y_tP_t$ and get
\begin{equation}
\label{eq:dYtPt}
\begin{aligned}
\d Y_tP_t=&P_t\d Y_t+Y_t\d P_t+ \wt{\E}\lt[\partial_{\mu}\sigma\lt(\PP_{X^*_t}, \wt{X}_t^*\rt)\wt{Y}_t\rt]\mathbb{D}^H_tP_t\d t\\
=&P_t\wt{\E}\lt[\partial_{\mu}\sigma\lt(\PP_{X^*_t}, \wt{X}_t^*\rt)\wt{Y}_t\rt]\d B_t^H\\
&+P_t\wt{\E}\lt[\lt\langle\partial_{\mu}b\lt(\PP_{(X_t^*,u^*_t)}, X^*_t, u^*_t, \wt{X}_t^*,\wt{u}^*_t\rt),\lt(\wt{Y}_t, \wt{u}_t-\wt{u}^*_t\rt)\rt\rangle\rt]\d t\\
&+P_t\partial_xb\lt(\PP_{(X^*_t,u^*_t)}, X_t^*, u^*_t\rt)Y_t\d t +P_t\partial_ub\lt(\PP_{(X^*_t,u^*_t)}, X_t^*, u^*_t\rt)(u_t-u_t^*)\d t\\
&+Y_t\alpha_t\d t +Y_t\beta_t\d W_t+\wt{\E}\lt[\partial_{\mu}\sigma\lt(\PP_{X^*_t}, \wt{X}_t^*\rt)\wt{Y}_t\rt]\mathbb{D}^H_tP_t\d t.
\end{aligned}
\end{equation}
Therefore, by integrating over the interval $[0,T]$ and considering that $Y_0=0$, we have
\begin{equation}
\label{eq:YtPt}
\begin{aligned}
Y_TP_T
=&\int_0^TP_s\wt{\E}\lt[\partial_{\mu}\sigma\lt(\PP_{X^*_s}, \wt{X}_s^*\rt)\wt{Y}_s\rt]\d B_s^H
+\int^T_0Y_s\beta_s\d W_s\\
&+\int^T_0 P_s\wt{\E}\lt[\lt\langle\partial_{\mu}b\lt(\PP_{(X_s^*,u^*_s)}, X^*_s, u^*_s, \wt{X}_s^*,\wt{u}^*_s\rt),\lt(\wt{Y}_s, \wt{u}_s-\wt{u}^*_s\rt)\rt\rangle\rt]\d s\\
&+\int^T_0 P_s\partial_xb\lt(\PP_{(X^*_s,u^*_s)}, X_s^*, u^*_s\rt)Y_s\d s \\
&+\int^t_0 P_s\partial_ub\lt(\PP_{(X^*_s,u^*_s)}, X_s^*, u^*_s\rt)(u_s-u_s^*)\d s\\
&+\int^T_0 Y_s\alpha_s\d s 
+\int^T_0\wt{\E}\lt[\partial_{\mu}\sigma\lt(\PP_{X^*_s}, \wt{X}_s^*\rt)\wt{Y}_s\rt]\mathbb{D}^H_sP_s\d s.
\end{aligned}
\end{equation}
By taking expectations with respect to $\PP$, we get
\begin{equation}
\label{eq:EYtPt}
\begin{aligned}
\E\lt[Y_TP_T\rt]
=&\int^T_0\E\lt[P_s\wt{\E}\lt[\lt\langle\partial_{\mu}b\lt(\PP_{(X_s^*,u^*_s)}, X^*_s, u^*_s, \wt{X}_s^*,\wt{u}^*_s\rt),\lt(\wt{Y}_s, \wt{u}_s-\wt{u}^*_s\rt)\rt\rangle\rt]\rt]\d s\\
&+\int^T_0\E\lt[ P_s\partial_xb\lt(\PP_{(X^*_s,u^*_s)}, X_s^*, u^*_s\rt)Y_s\rt]\d s \\
&+\int^T_0 \E\lt[P_s\partial_ub\lt(\PP_{(X^*_s,u^*_s)}, X_s^*, u^*_s\rt)(u_s-u_s^*)\rt]\d s\\
&+\int^T_0\E\lt[ Y_s\alpha_s\rt]\d s 
+\int^T_0\E\lt[\wt{\E}\lt[\partial_{\mu}\sigma\lt(\PP_{X^*_s}, \wt{X}_s^*\rt)\wt{Y}_s\rt]\mathbb{D}^H_sP_s\rt]\d s.
\end{aligned}
\end{equation}
From the definition of $(\wt{\Omega},\wt{\mathcal{F}},\wt{\PP})$, it follows that
\begin{equation}
\label{eq:EYtPt1}
\begin{aligned}
&\int^T_0\E\lt[P_s\wt{\E}\lt[\lt\langle\partial_{\mu}b\lt(\PP_{(X_s^*,u^*_s)}, X^*_s, u^*_s, \wt{X}_s^*,\wt{u}^*_s\rt),\lt(\wt{Y}_s, \wt{u}_s-\wt{u}^*_s\rt)\rt\rangle\rt]\rt]\d s\\
=&\int^T_0\E\lt[\wt{\E}\lt[\wt{P}_s\lt\langle\partial_{\mu}b\lt(\PP_{(X_s^*,u^*_s)},\wt{X}^*_s, \wt{u}^*_s, {X}_s^*,{u}^*_s\rt),\lt({Y}_s, {u}_s-{u}^*_s\rt)\rt\rangle\rt]\rt]\d s\\
=&\int^T_0\E\lt[\wt{\E}\lt[\wt{P}_s(\partial_{\mu}b)_1\lt(\PP_{(X_s^*,u^*_s)},\wt{X}^*_s, \wt{u}^*_s, {X}_s^*,u^*_s\rt){Y}_s\rt]\rt]\d s\\
&+\int^T_0\E\lt[\wt{\E}\lt[\wt{P}_s(\partial_{\mu}b)_2\lt(\PP_{(X_s^*,u^*_s)},\wt{X}^*_s, \wt{u}^*_s, X^*_s, {u}_s^*\rt)({u}_s-{u}^*_s)\rt]\rt]\d s.
\end{aligned}
\end{equation}
Moreover, for the latter term of equation (\ref{eq:EYtPt}), we have 
\begin{equation}
\label{eq:EYtPt2}
\begin{aligned}
&\int^T_0\E\lt[\wt{\E}\lt[\partial_{\mu}\sigma\lt(\PP_{X^*_s}, \wt{X}_s^*\rt)\wt{Y}_s\rt]\mathbb{D}^H_sP_s\rt]\d s
=\int^T_0\E\lt[\partial_{\mu}\sigma\lt(\PP_{X^*_s}, {X}_s^*\rt){Y}_s{\E}\lt[{\mathbb{D}}^H_s{P}_s\rt]\rt]\d s
\end{aligned}
\end{equation}

Substituting equations (\ref{eq:EYtPt1}) and (\ref{eq:EYtPt2}) in equation (\ref{eq:EYtPt}), we obtain
 \begin{equation}
 \label{eq:EYtPtfinal}
 \begin{aligned}
 &\E\lt[Y_TP_T\rt]\\
 =&\int^T_0\E\bigg[Y_s\bigg\{\wt{\E}\lt[\wt{P}_s(\partial_{\mu}b)_1\lt(\PP_{(X_s^*,u^*_s)},\wt{X}^*_s, \wt{u}^*_s, {X}_s^*\rt)\rt]+ P_s\partial_xb\lt(\PP_{(X^*_s,u^*_s)}, X_s^*, u^*_s\rt)\\
 &\qquad\qquad+\alpha_s+\partial_{\mu}\sigma\lt(\PP_{X^*_s}, {X}_s^*\rt){\E}\lt[{\mathbb{D}}^H_s{P}_s\rt]\bigg\}\bigg]\d s\\ &+\int^T_0\E\bigg[\bigg(\wt{\E}\lt[\wt{P}_s(\partial_{\mu}b)_2\lt(\PP_{(X_s^*,u^*_s)},\wt{X}^*_s, \wt{u}^*_s, {u}_s^*\rt)\rt]\\
 &\qquad +P_s\partial_ub\lt(\PP_{(X^*_s,u^*_s)}, X_s^*, u^*_s\rt)\bigg)({u}_s-{u}^*_s)\bigg]\d s.
 \end{aligned}
 \end{equation}
Now we substitute equation (\ref{eq:dJ1}) in (\ref{eq:EYtPtfinal}), and get
   \begin{equation}
   \label{eq:EYtPtfinal1}
   \begin{aligned}
  0\le 
&\int^T_0\E\bigg[Y_s\bigg\{\wt{\E}\lt[\wt{P}_s(\partial_{\mu}b)_1\lt(\PP_{(X_s^*,u^*_s)},\wt{X}^*_s, \wt{u}^*_s, {X}_s^*\rt)\rt]+ P_s\partial_xb\lt(\PP_{(X^*_s,u^*_s)}, X_s^*, u^*_s\rt)+\alpha_s\\
   &+\partial_{\mu}\sigma\lt(\PP_{X^*_s}, {X}_s^*\rt){\E}\lt[{\mathbb{D}}^H_s{P}_s\rt]+\partial_x f(\PP_{(X^*_s,u^*_s)},X^*_s,u^*_s)\\
   &+\wt{\E}\lt[(\partial_\mu f)_1( \PP_{(X^*_s,u^*_s)},\wt{X}^*_s,\wt{u}^*_s,X^*_s,u^*_s)\rt]\bigg\}\bigg]\d s\\
   &+\int^T_0\E\bigg[(u_s-{u}^*_s)\bigg\{\wt{\E}\lt[(\partial_\mu f)_2( \PP_{(X^*_s,u^*_s)},\wt{X}^*_s,\wt{u}^*_s,X^*_s,u^*_s)\rt]+\partial_u f( \PP_{(X^*_s,u^*_s)},X^*_s,u^*_s)\\
   &+\wt{\E}\lt[\wt{P}_s(\partial_{\mu}b)_2\lt(\PP_{(X_s^*,u^*_s)},\wt{X}^*_s, \wt{u}^*_s, X^*_s, {u}_s^*\rt)\rt] +P_s\partial_ub\lt(\PP_{(X^*_s,u^*_s)}, X_s^*, u^*_s\rt)\bigg\}\bigg]\d s.\\
   \end{aligned}
   \end{equation}
    Letting the first integral, which integrand contains $Y_s$, equal to zero, we get
    \begin{equation}
    \label{eq:alpha}
    \begin{aligned}
    \alpha_s=&-\wt{\E}\lt[\wt{P}_s(\partial_{\mu}b)_1\lt(\PP_{(X_s^*,u^*_s)},\wt{X}^*_s, \wt{u}^*_s, {X}_s^*\rt)\rt]- P_s\partial_xb\lt(\PP_{(X^*_s,u^*_s)}, X_s^*, u^*_s\rt)\\
    &-\partial_{\mu}\sigma\lt(\PP_{X^*_s}, {X}_s^*\rt){\E}\lt[{\mathbb{D}}^H_s{P}_s\rt]-\partial_x f(\PP_{(X^*_s,u^*_s)},X^*_s,u^*_s)\\
    &-\wt{\E}\lt[(\partial_\mu f)_1( \PP_{(X^*_s,u^*_s)},\wt{X}^*_s,\wt{u}^*_s,X^*_s,u^*_s)\rt].
    \end{aligned}
    \end{equation}
    This gives the following form of the BSDE for $ P_t=P_T-\int^T_t\alpha_s\d s -\int^T_t \beta_s\d W_s$:
    \begin{equation}
    \label{eq:Pt}
    \begin{aligned}
    P_t =&P_T+\int^T_t\bigg\{\wt{\E}\lt[\wt{P}_s(\partial_{\mu}b)_1\lt(\PP_{(X_s^*,u^*_s)},\wt{X}^*_s, \wt{u}^*_s, {X}_s^*,u^*_s\rt)\rt]+ P_s\partial_xb\lt(\PP_{(X^*_s,u^*_s)}, X_s^*, u^*_s\rt)\\
    &\qquad\qquad\quad+\partial_{\mu}\sigma\lt(\PP_{X^*_s}, {X}_s^*\rt){\E}\lt[{\mathbb{D}}^H_s{P}_s\rt]+\partial_x f(\PP_{(X^*_s,u^*_s)},X^*_s,u^*_s)\\
    &\qquad\qquad\quad+\wt{\E}\lt[(\partial_\mu f)_1( \PP_{(X^*_s,u^*_s)},\wt{X}^*_s,\wt{u}^*_s,X^*_s,u^*_s)\rt]\bigg\}\d s -\int^T_t \beta_s\d W_s,
    \end{aligned}
    \end{equation}
    which is a mean-field BSDE driven by the standard Brownian motion $W$. Such kind of mean-field BSDE (without the term of Malliavin derivative) was studied firstly by Buckdahn et al. \cite{BDLP} \cite{BLP}.  We recall again that in the above BSDE, the expectation $\wt{\E}$ only concerns the processes with tildes.
    
    We suppose that there exists such a solution $(P,\beta)$ of (\ref{eq:Pt}); its existence and uniqueness will be discussed later for a special case.
   
   With this choice of $P_t$, equation (\ref{eq:EYtPtfinal1}) now becomes
\begin{equation}
   \begin{aligned}
   0\le&\int^T_0\E\bigg[(u_s-{u}^*_s)\bigg\{\wt{\E}\lt[(\partial_\mu f)_2( \PP_{(X^*_s,u^*_s)},\wt{X}^*_s,\wt{u}^*_s,X^*_s,u^*_s)\rt]+\partial_u f( \PP_{(X^*_s,u^*_s)},X^*_s,u^*_s)\\
   &+\wt{\E}\lt[\wt{P}_s(\partial_{\mu}b)_2\lt(\PP_{(X_s^*,u^*_s)},\wt{X}^*_s, \wt{u}^*_s, X^*_s, {u}_s^*\rt)\rt] +P_s\partial_ub\lt(\PP_{(X^*_s,u^*_s)}, X_s^*, u^*_s\rt)\bigg\}\bigg]\d s.\\
\end{aligned}
\end{equation}   
   From the fact that $U$ is open and from the arbitrariness of $u\in \mathcal{U}([0,T])$, we have 
   \begin{equation}
   \label{eq:hamiltonian}
   \begin{aligned}
   0=&\wt{\E}\lt[(\partial_\mu f)_2( \PP_{(X^*_t,u^*_t)},\wt{X}^*_t,\wt{u}^*_t,X^*_t,u^*_t)\rt]+\partial_u f( \PP_{(X^*_t,u^*_t)},X^*_t,u^*_t)\\
   &+\wt{\E}\lt[\wt{P}_t(\partial_{\mu}b)_2\lt(\PP_{(X_t^*,u^*_t)},\wt{X}^*_t, \wt{u}^*_t, X^*_t, {u}_t^*\rt)\rt]+ P_t\partial_ub\lt(\PP_{(X^*_t,u^*_t)}, X^*_t, u^*_t\rt),
   \end{aligned}
   \end{equation}
  $\d \PP\textrm{-a.s.}, \d t\textrm{-a.e}.$ 
  
  Now we can conclude the above calculations in the following necessary conditions of  Pontryagin-type maximum principle, which is our main result.
  \begin{theorem}
  	\label{thm-main}
  	If $(X^*,u^*)$ is an optimal pair of mean-field stochastic control problem $(\ref{eq:control})-(\ref{eq:optimal cost})$, then $(X^*,u^*)$ satisfies the following system:
    \begin{equation}
    \label{eq:necessary}
    \begin{cases}
    \begin{aligned}
    	X_t^*=&x+\int^t_0\sigma(\PP_{X_s^*})\d B^H_s+\int^t_0 b(\PP_{(X_s^*, u_s^*)},X_s^*, u_s^*)\d s,\\ 
    	P_T=&\partial_x g(X^*_T,\PP_{X^*_T})+\wt{\E}\lt[\partial_\mu g(\wt{X^*_T},\PP_{X^*_T},X^*_T)\rt],\\
    	P_t =&P_T-\int^T_t \beta_s\d W_s+\int^T_t\bigg\{\wt{\E}\lt[\wt{P}_s(\partial_{\mu}b)_1\lt(\PP_{(X_s^*,u^*_s)},\wt{X}^*_s, \wt{u}^*_s, {X}_s^*,u^*_s\rt)\rt]\\&+ P_s\partial_xb\lt(\PP_{(X^*_s,u^*_s)}, X_s^*, u^*_s\rt)+\partial_{\mu}\sigma\lt(\PP_{X^*_s}, {X}_s^*\rt){\E}\lt[{\mathbb{D}}^H_s{P}_s\rt]\\
    	&+\partial_x f(\PP_{(X^*_s,u^*_s)},X^*_s,u^*_s)+\wt{\E}\lt[(\partial_\mu f)_1( \PP_{(X^*_s,u^*_s)},\wt{X}^*_s,\wt{u}^*_s,X^*_s,u^*_s)\rt]\bigg\}\d s ,\\
    	0=&\wt{\E}\lt[\wt{P}_t(\partial_{\mu}b)_2\lt(\PP_{(X_t^*,u^*_t)},\wt{X}^*_t, \wt{u}^*_t, X^*_t, {u}_t^*\rt)\rt]+ P_t\partial_ub\lt(\PP_{(X^*_t,u^*_t)}, X^*_t, u^*_t\rt)\\
    	&+\wt{\E}\lt[(\partial_\mu f)_2( \PP_{(X^*_t,u^*_t)},\wt{X}^*_t,\wt{u}^*_t,X^*_t,u^*_t)\rt]+\partial_u f( \PP_{(X^*_t,u^*_t)},X^*_t,u^*_t),\\
    	&\d \PP\textrm{-a.s.}, \d t\textrm{-a.e}.
    \end{aligned}
    \end{cases}  	
    \end{equation}
  
  \end{theorem}

We can also give a sufficient condition for optimality under some more assumptions. In the following we define our Hamiltonian, for $(\mu,x,u,y,z)\in\mathcal{P}_2(\R\times\R^m)\times\R\times\R^d\times\R\times\R$,  $$H(\mu,x,u,y,z):=f(\mu,x,u)+b(\mu,x,u)y+\sigma(\mu)z.$$ 
For the following assumption, we recall the definition of joint convexity (\ref{convexity}).

\textbf{(H6)}. $g:\R\times\mathcal{P}_2(\R)\to \R$ is jointly convex in $(x,\mu)$. The Hamiltonian $H(\mu,x,u,$ $y,z)$ is jointly convex in $(\mu,x,u)$.  

\begin{theorem}
\label{thm:suffi}
Suppose (H2)-(H6) hold. Let $(u^*_t,X^*_t)_{t\in[0,T]}$ satisfy system (\ref{eq:necessary}). Then $(u^*_t,X^*_t)_{t\in[0,T]}$ is  optimal and $J(u^*)=\inf_{u\in \mathcal{U}([0,T])} J(u)$.
\end{theorem}
\begin{proof}
	Suppose $(u_t,X^u_t)$ is an arbitrary control and the corresponding sate. Then from the definition of cost functional $J(u)$, we have 
	\begin{equation}
	\label{eq:Ju*-Ju}
	\begin{aligned}
	J(u^*)-J(u)=&\E\lt[g(X^{*}_T,\PP_{X^*_T})-g(X^{u}_T,\PP_{X^u_T})\rt]\\
	&+\E\lt[\int^T_0\lt(f(\PP_{(X^*_t,u^*_t)},X^*_t,u^*_t)-f(\PP_{(X^u_t,u_t)},X^u_t,u_t)\rt)\d t\rt].
    \end{aligned}
	\end{equation}
	Hence from the joint convexity of $g$, we have
	\begin{equation}
	\label{eq:Ju*-Ju1}
	\begin{aligned}
	J(u^*)-J(u)
	\le&\E\lt[\lt(\partial_x g(X^{*}_T,\PP_{X^*_T})+\wt{\E}\lt[\partial_\mu g(\wt{X}^{*}_T,\PP_{X^*_T},{X}^*_T)\rt]\rt)(X^*_T-X^u_T)\rt]\\
	&+\E\lt[\int^T_0\lt(f(\PP_{(X^*_t,u^*_t)},X^*_t,u^*_t)-f(\PP_{(X^u_t,u_t)},X^u_t,u_t)\rt)\d t\rt]\\
	=&\E\lt[P_T(X^*_T-X^u_T)\rt]\\
	&+\E\lt[\int^T_0\lt(f(\PP_{(X^*_t,u^*_t)},X^*_t,u^*_t)-f(\PP_{(X^u_t,u_t)},X^u_t,u_t)\rt)\d t\rt].
	\end{aligned}
	\end{equation}
	By applying Ito's formula (equation (\ref{eq:ito_WH})) to $P_t(X^*_t-X^u_t)$ and taking the expectation on both sides, we get:
	\begin{equation}
	\label{eq:Ju*-Ju2}
	\begin{aligned}
	&\E\lt[P_T(X^*_T-X^u_T)\rt]+\E\lt[\int^T_0\lt(f(\PP_{(X^*_s,u^*_s)},X^*_s,u^*_s)-f(\PP_{(X^u_s,u_s)},X^u_s,u_s)\rt)\d s\rt]\\
	=&\int^T_0\E\lt[P_s\lt(b(\PP_{(X^*_s,u^*_s)},X^*_s,u^*_s)-b(\PP_{(X^u_s,u_s)},X^u_s,u_s)\rt)\rt]\d s\\
	&-\int^T_0\E\lt[{P}_s\wt{\E}\lt[(\partial_{\mu}b)_1\lt(\PP_{(\wt{X}_s^*,\wt{u}^*_s)},{X}^*_s, {u}^*_s,\wt{X}^*_s, \wt{u}^*_s\rt)(\wt{X}^*_s-\wt{X}^u_s)\rt]\rt]\d s\\
	&-\int^T_0\E\lt[ P_s\partial_xb\lt(\PP_{(X^*_s,u^*_s)}, X_s^*, u^*_s\rt)({X}^*_s-{X}^u_s)\rt]\d s\\
	&+\int^T_0\lt(\sigma\lt(\PP_{X^*_s}\rt)- \sigma\lt(\PP_{{X}^u_s}\rt)\rt)\E\lt[\mathbb{D}^H_sP_s\rt]\d s\\
	&-\int^T_0\E\lt[\partial_{\mu}\sigma\lt(\PP_{X^*_s}, {X}_s^*\rt)({X}^*_s-{X}^u_s)\rt]{\E}\lt[{\mathbb{D}}^H_s{P}_s\rt]\d s\\
	&-\int^T_0\E\lt[\partial_x f(\PP_{(X^*_s,u^*_s)},X^*_s,u^*_s)({X}^*_s-{X}^u_s)\rt]\d s\\
	&-\int^T_0\E\lt[\wt{\E}\lt[(\partial_\mu f)_1( \PP_{(X^*_s,u^*_s)},X^*_s,u^*_s,\wt{X}^*_s,\wt{u}^*_s)(\wt{X}^*_s-\wt{X}^u_s)\rt]\rt]\d s.
	\end{aligned}
	\end{equation}
	From the joint convexity of Hamiltonian $H$ we get 
	 \begin{equation}
	 \label{eq:Ju*-Ju3}
	 \begin{aligned}
	 &P_s\lt(b(\PP_{(X^*_s,u^*_s)},X^*_s,u^*_s)-b(\PP_{(X^u_s,u_s)},X^u_s,u_s)\rt)
	 +\lt(\sigma\lt(\PP_{X^*_s}\rt)- \sigma\lt(\PP_{{X}^u_s}\rt)\rt)\E\lt[\mathbb{D}^H_sP_s\rt]\\
	 &+f(\PP_{(X^*_s,u^*_s)},X^*_s,u^*_s)-f(\PP_{(X^u_s,u_s)},X^u_s,u_s)\\
	 \le&{P}_s\wt{\E}\lt[(\partial_{\mu}b)_1\lt(\PP_{({X}^*_s, {u}^*_s)},{X}^*_s, {u}^*_s,\wt{X}^*_s, \wt{u}^*_s\rt)(\wt{X}^*_s-\wt{X}^u_s)\rt]\\
	 &+ P_s\partial_xb\lt(\PP_{(X^*_s,u^*_s)}, X_s^*, u^*_s\rt)({X}^*_s-{X}^u_s)\\
	 &+{P}_s\wt{\E}\lt[(\partial_{\mu}b)_2\lt(\PP_{({X}^*_s, {u}^*_s)},{X}^*_s, {u}^*_s,\wt{X}^*_s, \wt{u}^*_s\rt)(\wt{u}^*_s-\wt{u}_s)\rt]\\
	 &+ P_s\partial_ub\lt(\PP_{(X^*_s,u^*_s)}, X_s^*, u^*_s\rt)(u^*_s-u_s)\\
	 &+\E\lt[\partial_{\mu}\sigma\lt(\PP_{X^*_s},{X}_s^*\rt)({X}^*_s-{X}^u_s)\rt]{\E}\lt[{\mathbb{D}}^H_s{P}_s\rt]\\
	 &+\partial_x f(\PP_{(X^*_s,u^*_s)},X^*_s,u^*_s)({X}^*_s-{X}^u_s)\\
	 &+\wt{\E}\lt[(\partial_\mu f)_1( \PP_{(X^*_s,u^*_s)},X^*_s,u^*_s,\wt{X}^*_s,\wt{u}^*_s)(\wt{X}^*_s-\wt{X}^u_s)\rt]\\
	 &+\wt{\E}\lt[(\partial_{\mu}f)_2\lt(\PP_{({X}^*_s, {u}^*_s)},{X}^*_s, {u}^*_s,\wt{X}^*_s, \wt{u}^*_s\rt)(\wt{u}^*_s-\wt{u}_s)\rt]\\
	 &+\partial_u f\lt(\PP_{(X^*_s,u^*_s)}, X_s^*, u^*_s\rt)(u^*_s-u_s),\\
	 \end{aligned}
	 \end{equation}
	 where due to equation (\ref{eq:necessary}), 
	 \begin{equation}
	 \label{eq:condition-hamilton}
	 \begin{aligned}
	 &\E\lt[\wt{\E}\lt[(\partial_{\mu}f)_2\lt(\PP_{({X}^*_s, {u}^*_s)},{X}^*_s, {u}^*_s,\wt{X}^*_s, \wt{u}^*_s\rt)(\wt{u}^*_s-\wt{u}_s)\rt]\rt] \\
	 &+\E\lt[\partial_u f\lt(\PP_{(X^*_s,u^*_s)}, X_s^*, u^*_s\rt)(u^*_s-u_s)\rt]\\
	 &+\E\bigg[{P}_s\wt{\E}\lt[(\partial_{\mu}b)_2\lt(\PP_{({X}_s^*,{u}^*_s)},{X}^*_s, {u}^*_s,\wt{X}^*_s, \wt{u}^*_s\rt)(\wt{u}^*_s-\wt{u}_s)\rt]\\
	 &\quad+ P_s\partial_ub\lt(\PP_{(X^*_s,u^*_s)}, X_s^*, u^*_s\rt)(u^*_s-u_s)\bigg]\\
	 =\E&\bigg[\bigg(\wt{\E}\lt[(\partial_{\mu}f)_2\lt(\PP_{({X}^*_s, {u}^*_s)},\wt{X}^*_s, \wt{u}^*_s,{X}^*_s, {u}^*_s\rt)\rt]+\partial_u f\lt(\PP_{(X^*_s,u^*_s)}, X_s^*, u^*_s\rt)\\
	 &+\wt{P}_s(\partial_{\mu}b)_2\lt(\PP_{({X}^*_s, {u}^*_s)},\wt{X}^*_s, \wt{u}^*_s,{X}^*_s, {u}^*_s\rt)+ P_s\partial_ub\lt(\PP_{(X^*_s,u^*_s)}, X_s^*, u^*_s\rt)\bigg)(u^*_s-u_s)\bigg]\\
	 =&0.
	 \end{aligned}
	 \end{equation} 
	 Therefore we get from the equations (\ref{eq:Ju*-Ju})-(\ref{eq:Ju*-Ju3}) that 
	 $$
	 J(u^*)-J(u)\le 0,
	 $$
	 which means $(u^*,X^*)$ is an optimal pair.
\end{proof}

\begin{remark}
	We emphasise that in (H6), similarly to Carmona and Delarue \cite{CaDe}, we assume the convexity of the Hamiltonian $H$.  If there is no running cost function $f$, supposing convexity is in some sense equivalent to assuming linearity, because of the  multiplications  with $P_t$ and $\E\lt[\mathbb{D}^H_tP_t\rt]$, respectively, which sign can change. With the assumption of linearity, the inequality in   (\ref{eq:Ju*-Ju3}) become equality. 
\end{remark}

In the following we give another sufficient condition which allows to have more general coefficients which are not necessarily linear. For this we need the following assumption and we recall the definition of strict concavity (\ref{strictconvexity}).

\textbf{(H7)}. Then function $g:\R\times\mathcal{P}_2(\R)\to \R$ is jointly convex in $(x,\mu)$, with $\partial_x g\ge 0$ and $\partial _\mu g\ge 0$, and  $b(\eta,x,u):\mathcal{P}_2(\R\times U)\times\R\times U\to \R$ is jointly convex in $(\eta,x,u)$ with $(\partial_\mu b)_1(\eta,x,u,y)\ge 0$ and strictly convex in $(\mu,u)$. Moreover,  $f(\eta,x,u):\mathcal{P}_2(\R\times U)\times\R\times U\to \R$ is jointly convex in $(\eta,x,u)$ and strictly convex in $(\mu,u)$, with $(\partial_\mu f)_1(\eta,x,u,y)\ge 0$, $\partial_x f(\eta,x,u,y)\ge 0$, and $\sigma(\mu)\equiv \sigma\in\R$.

\begin{theorem}
	\label{thm:suffi1}
	Suppose (H2)-(H5) and (H7) hold. Let $(u^*_t,X^*_t)_{t\in[0,T]}$ satisfy system (\ref{eq:necessary}). Then $(u^*_t,X^*_t)_{t\in[0,T]}$ is  optimal  and $J(u^*)=\inf_{u\in \mathcal{U}([0,T])} J(u)$.
\end{theorem}

\begin{remark}
	The existence of such functions can be easily verified. For instance, we can choose $g$ as an increasing convex function with bounded derivatives. For an increasing joint convex  function  $\bar{b}: \R\times U\times\R\times U$ with bounded derivatives, we can choose 
	$b\lt(\PP_{(X_s^*,u^*_s)},{X}^*_s, {u}^*_s\rt):=\wt{\E}\lt[\bar{b}\lt(\wt{X}^*_s, \wt{u}^*_s,X_s^*,u^*_s\rt)\rt]$.
\end{remark}

\begin{proof}
	With these assumptions, the BSDE in system (\ref{eq:necessary}) now becomes
	\begin{equation}
	\label{BSDE}
	\begin{aligned}
	P_t=&\partial_x g(X^*_T,\PP_{X^*_T})+\wt{\E}\lt[\partial_\mu g(\wt{X^*_T},\PP_{X^*_T},X^*_T)\rt]-\int^T_t \beta_s\d W_s\\
	&+\int^T_t\wt{\E}\lt[\wt{P}_s(\partial_{\mu}b)_1\lt(\PP_{(X_s^*,u^*_s)},\wt{X}^*_s, \wt{u}^*_s, {X}_s^*,u^*_s\rt)\rt]\d s\\	&+\int^T_t P_s\partial_xb\lt(\PP_{(X^*_s,u^*_s)}, X_s^*, u^*_s\rt)\d s \\
	&+\int^T_t\lt[\partial_x f(\PP_{(X^*_s,u^*_s)},X^*_s,u^*_s)+\wt{\E}\lt[(\partial_\mu f)_1( \PP_{(X^*_s,u^*_s)},\wt{X}^*_s,\wt{u}^*_s,X^*_s,u^*_s)\rt]\rt]\d s,\\
	\end{aligned}
	\end{equation}
	which is a mean-field BSDE in the classical sense which was studied by Buckdahn et. al. \cite{BLP}.   We compare it with the following BSDE
	\begin{equation}
	\label{BSDE0}
	\begin{aligned}
	Q_t=&0
	+\int^T_t\wt{\E}\lt[\wt{Q}_s(\partial_{\mu}b)_1\lt(\PP_{(X_s^*,u^*_s)},\wt{X}^*_s, \wt{u}^*_s, {X}_s^*,u^*_s\rt)\rt]\d s \\&	+\int^T_t Q_s\partial_xb\lt(\PP_{(X^*_s,u^*_s)}, X_s^*, u^*_s\rt)\d s -\int^T_t Z_s\d W_s,\\
	\end{aligned}
	\end{equation}
	which has a unique solution $(Q_t,Z_t)=(0,0)$. From the comparison result (Theorem 3.2) in \cite{BLP}, we see that
	$P_t\ge Q_t=0,$ for all $t\in[0,T]$. On the other hand, due to $(H3)$-$H(5)$ and $(H7)$, we have $0\le \partial_x g\le C$, $0\le \partial_\mu g\le C$, $0\le (\partial_\mu b)_1\le C$ and $|\partial_x b|\le C$. We compare equation (\ref{BSDE}) again with the following BSDE
		\begin{equation}
		\label{BSDE1}
		\begin{aligned}
		Q^\pr_t=&C
		+C\int^T_t\E\lt[{Q}^\pr_s\rt]\d s+C\int^T_t Q^\pr_s\d s -\int^T_t Z^\pr_s\d W_s,
		\end{aligned}
		\end{equation}
	which has a unique solution $(Q^\pr_t,Z^\pr_t)=(C\exp\{2C(T-t)\},0), t\in[0,T]$. From the comparison result again, we get $P_t\le Q^\pr_t$, hence $P_t$ is uniformly bounded. The equation (\ref{eq:Ju*-Ju2}) becomes
	\begin{equation}
	\label{eq:Ju*-Ju21}
	\begin{aligned}
	&\E\lt[P_T(X^*_T-X^u_T)\rt]+\E\lt[\int^T_0\lt(f(\PP_{(X^*_s,u^*_s)},X^*_s,u^*_s)-f(\PP_{(X^u_s,u_s)},X^u_s,u_s)\rt)\d s\rt]\\
	=&\E\lt[\int^T_0\lt(f(\PP_{(X^*_s,u^*_s)},X^*_s,u^*_s)-f(\PP_{(X^u_s,u_s)},X^u_s,u_s)\rt)\d s\rt]\\
	&+\int^T_0\E\lt[P_s\lt(b(\PP_{(X^*_s,u^*_s)},X^*_s,u^*_s)-b(\PP_{(X^u_s,u_s)},X^u_s,u_s)\rt)\rt]\d s\\
	&-\int^T_0\E\lt[{P}_s\wt{\E}\lt[(\partial_{\mu}b)_1\lt(\PP_{(\wt{X}_s^*,\wt{u}^*_s)},{X}^*_s, {u}^*_s,\wt{X}^*_s, \wt{u}^*_s\rt)(\wt{X}^*_s-\wt{X}^u_s)\rt]\rt]\d s\\
	&-\int^T_0\E\lt[ P_s\partial_xb\lt(\PP_{(X^*_s,u^*_s)}, X_s^*, u^*_s\rt)({X}^*_s-{X}^u_s)\rt]\d s\\
		&-\int^T_0\E\lt[\partial_x f(\PP_{(X^*_s,u^*_s)},X^*_s,u^*_s)({X}^*_s-{X}^u_s)\rt]\d s\\
		&-\int^T_0\E\lt[\wt{\E}\lt[(\partial_\mu f)_1( \PP_{(X^*_s,u^*_s)},X^*_s,u^*_s,\wt{X}^*_s,\wt{u}^*_s)(\wt{X}^*_s-\wt{X}^u_s)\rt]\rt]\d s..
	\end{aligned}
	\end{equation}
	We deduce from the joint convexity of $b, f$ and the positivity of $P_s$  that
	\begin{equation}
	\label{eq:Ju*-Ju31}
	\begin{aligned}
	&f(\PP_{(X^*_s,u^*_s)},X^*_s,u^*_s)-f(\PP_{(X^u_s,u_s)},X^u_s,u_s)\\
	&+P_s\lt(b(\PP_{(X^*_s,u^*_s)},X^*_s,u^*_s)-b(\PP_{(X^u_s,u_s)},X^u_s,u_s)\rt)\\
	\le&\partial_x f(\PP_{(X^*_s,u^*_s)},X^*_s,u^*_s)({X}^*_s-{X}^u_s)+\wt{\E}\lt[(\partial_\mu f)_1( \PP_{(X^*_s,u^*_s)},X^*_s,u^*_s,\wt{X}^*_s,\wt{u}^*_s)(\wt{X}^*_s-\wt{X}^u_s)\rt]\\ &+\wt{\E}\lt[(\partial_{\mu}f)_2\lt(\PP_{({X}^*_s, {u}^*_s)},{X}^*_s, {u}^*_s,\wt{X}^*_s, \wt{u}^*_s\rt)(\wt{u}^*_s-\wt{u}_s)\rt]\\
	&+\partial_u f\lt(\PP_{(X^*_s,u^*_s)}, X_s^*, u^*_s\rt)(u^*_s-u_s)\\
	&+{P}_s\wt{\E}\lt[(\partial_{\mu}b)_1\lt(\PP_{({X}_s^*,{u}^*_s)},{X}^*_s, {u}^*_s,\wt{X}^*_s, \wt{u}^*_s\rt)(\wt{X}^*_s-\wt{X}^u_s)\rt]\\
	&+ P_s\partial_xb\lt(\PP_{(X^*_s,u^*_s)}, X_s^*, u^*_s\rt)({X}^*_s-{X}^u_s)\\
	&+{P}_s\wt{\E}\lt[(\partial_{\mu}b)_2\lt(\PP_{({X}_s^*,{u}^*_s)},{X}^*_s, {u}^*_s,\wt{X}^*_s, \wt{u}^*_s\rt)(\wt{u}^*_s-\wt{u}_s)\rt]\\
	&+ P_s\partial_ub\lt(\PP_{(X^*_s,u^*_s)}, X_s^*, u^*_s\rt)(u^*_s-u_s),
	\end{aligned}
	\end{equation}
	 where again due to equation (\ref{eq:necessary}), for almost all $s\in[0,T]$, we have equation (\ref{eq:condition-hamilton}).
	 Hence from equations (\ref{eq:Ju*-Ju1}), (\ref{eq:Ju*-Ju21}) and (\ref{eq:Ju*-Ju31}), we get 
	 $$
	 J(u^*)-J(u)\le 0.
	 $$
	 Therefore, the optimality of $(u^*,X^*)$ follows.
\end{proof}

Concerning the solvability of system (\ref{eq:necessary}) under the conditions of Theorem \ref{thm:suffi1}, we proceed as follows.
For any given $(P,\xi)\in L^2(\mathcal{F}_t)\times L^2(\mathcal{F}_t)$, we suppose that there is some $\eta\in L^2(\mathcal{F}_t;U)$  such that:
\begin{equation}
\label{eq:solvability}
\begin{aligned}
0=&\wt{\E}\lt[(\partial_\mu f)_2\lt(\PP_{(\xi,\eta)},\wt{\xi},\wt{\eta},\xi,\eta\rt)\rt]+\partial_u f(\PP_{(\xi,\eta)},\xi,\eta)\\
&+
\wt{\E}\lt[\wt{P}(\partial_\mu b)_2\lt(\PP_{(\xi,\eta)},\wt{\xi},\wt{\eta},\xi,\eta\rt)\rt]+P(\partial_u b)(\PP_{(\xi,\eta)},\xi,\eta).
\end{aligned}
\end{equation} 
\begin{lemma}
	The mapping $(P,\xi)\to \eta$ is Lipschitz under $L^2$-norm.
\end{lemma}
\begin{proof}
Given $(P,\xi)\in L^2(\mathcal{F}_t)\times L^2(\mathcal{F}_t)$ and  $(\hat{P},\hat{\xi})\in L^2(\mathcal{F}_t)\times L^2(\mathcal{F}_t)$, let $\eta\in L^2(\mathcal{F}_t,U)$ be the solution of (\ref{eq:solvability}) associated with $(P,\xi)$ and $\hat{\eta}\in L^2(\mathcal{F}_t,U)$ be that for  $(\hat{P},\hat{\xi})$. Then the strict convexity of $b$ and $f$ allows us to show, there exists  $\lambda>0$, such that
\begin{equation*}
\begin{aligned}
&\lambda\E\lt[ |\hat{\eta}-\eta|^2\rt]\\
\le &\E\lt[\wt{\E}\lt[\lt(P(\partial_\mu b)_2(\PP_{(\xi,\hat{\eta})},\xi,\hat{\eta},\wt{\xi},\wt{\hat{\eta}})-P(\partial_\mu b)_2(\PP_{(\xi,\eta)},\xi,\eta,\wt{\xi},\wt{\eta})\rt)(\wt{\hat{\eta}}-\wt{\eta})\rt]\rt]\\
&+\E\lt[\lt(P\partial_u b(\PP_{(\xi,\hat{\eta})},\xi,\hat{\eta})-P\partial_u b(\PP_{(\xi,\eta)},\xi,\eta)\rt)({\hat{\eta}}-{\eta})\rt]\\
&+\E\lt[\wt{\E}\lt[\lt((\partial_\mu f)_2(\PP_{(\xi,\hat{\eta})},\xi,\hat{\eta},\wt{\xi},\wt{\hat{\eta}})-(\partial_\mu f)_2(\PP_{(\xi,\eta)},\xi,\eta,\wt{\xi},\wt{\eta})\rt)(\wt{\hat{\eta}}-\wt{\eta})\rt]\rt]\\
&+\E\lt[\lt(\partial_u f(\PP_{(\xi,\hat{\eta})},\xi,\hat{\eta})-\partial_u f(\PP_{(\xi,\eta)},\xi,\eta)\rt)({\hat{\eta}}-{\eta})\rt]\\
=& \E\lt[\wt{\E}\lt[\lt(P(\partial_\mu b)_2(\PP_{(\xi,\hat{\eta})},\xi,\hat{\eta},\wt{\xi},\wt{\hat{\eta}})-\hat{P}(\partial_\mu b)_2(\PP_{(\hat{\xi},\hat{\eta})},\hat{\xi},\hat{\eta},\wt{\hat{\xi}},\wt{\hat{\eta}})\rt)(\wt{\hat{\eta}}-\wt{\eta})\rt]\rt]\\
&+\E\lt[\lt(P\partial_u b(\PP_{(\xi,\hat{\eta})},\xi,\hat{\eta})-\hat{P}\partial_u b(\PP_{(\hat{\xi},\hat{\eta})},\hat{\xi},\hat{\eta})\rt)({\hat{\eta}}-{\eta})\rt]\\
&+\E\lt[\wt{\E}\lt[\lt((\partial_\mu f)_2(\PP_{(\xi,\hat{\eta})},\xi,\hat{\eta}),\wt{\xi},\wt{\hat{\eta}})-(\partial_\mu f)_2(\PP_{(\hat{\xi},\hat{\eta})},\hat{\xi},\hat{\eta},\wt{\hat{\xi}},\wt{\hat{\eta}})\rt)(\wt{\hat{\eta}}-\wt{\eta})\rt]\rt]\\
&+\E\lt[\lt(\partial_u f(\PP_{(\xi,\hat{\eta})},\xi,\hat{\eta})-\partial_u f(\PP_{(\hat{\xi},\hat{\eta})},\hat{\xi},\hat{\eta})\rt)({\hat{\eta}}-{\eta})\rt].
\end{aligned}
\end{equation*}
Then we get from the boundedness of $P$ (due to the proof of Theorem \ref{thm:suffi1}) and the Lipschitz continuity of the derivatives of $b$ and $f$ (see (H4)) that
\begin{equation*}
\begin{aligned}
\lambda\E\lt[ |\hat{\eta}-\eta|^2\rt]
\le &C\lt(\E\lt[|\hat{\eta}-\eta|^2\rt]\rt)^{1/2}\lt(\lt(\E\lt[|\hat{P}-P|^2\rt]\rt)^{1/2}+\lt(\E\lt[|\hat{\xi}-\xi|^2\rt]\rt)^{1/2}\rt).
\end{aligned}
\end{equation*}
Thus, we have 
$$
\E\lt[|\hat{\eta}-\eta|^2\rt]\le C\lt(\E\lt[|\hat{\xi}-\xi|^2+|\hat{P}-P|^2\rt]\rt).
$$
Hence,
the mapping $(P,\xi)\to \eta$ is Lipschitz in the $L^2$-norm, i.e.,  there exists a Lipschitz function $\eta: L^2(\mathcal{F}_t)\times L^2(\mathcal{F}_t)\to L^2(\mathcal{F}_t;U)$ such that $\eta=\eta(P,\xi)$ solves (\ref{eq:solvability}). 	
\end{proof}

With the above lemma, we know, in particular, that the solution of $(\ref{eq:solvability})$ is unique. Thus, the optimal control $u^*_t$, if it exists,  must satisfy the relation:  $u^*_t=\eta(P_t,X^*_t)$.  
Therefore, with such $u^*_t$, under the conditions of Theorem \ref{thm:suffi1}, the system (\ref{eq:necessary}) becomes
\begin{equation}
\label{eq:necessary1}
\begin{cases}
\begin{aligned}
X_t^*=&x+\sigma B^H_t+\int^t_0 b(\PP_{(X_s^*, \eta(P_s,X^*_s))},X_s^*, \eta(P_s,X^*_s))\d s,\\ 
P_t =&\partial_x g(X^*_T,\PP_{X^*_T})+\wt{\E}\lt[\partial_\mu g(\wt{X^*_T},\PP_{X^*_T},X^*_T)\rt]-\int^T_t \beta_s\d W_s\\
&+\int^T_t\bigg\{\wt{\E}\lt[\wt{P}_s(\partial_{\mu}b)_1\lt(\PP_{(X_s^*,\eta(P_s,X^*_s))},\wt{X}^*_s, \eta(\wt{P}_s,\wt{X}^*_s), {X}_s^*,\eta(P_s,X^*_s)\rt)\rt]\\&+ P_s\partial_xb\lt(\PP_{(X^*_s,\eta(P_s,X^*_s))}, X_s^*, \eta(P_s,X^*_s)\rt)\\
&+\partial_x f(\PP_{(X^*_s,\eta(P_s,X^*_s))},X^*_s,\eta(P_s,X^*_s))\\
&+\wt{\E}\lt[(\partial_\mu f)_1( \PP_{(X^*_s,\eta(P_s,X^*_s))},\wt{X}^*_s,\eta(\wt{P}_s,\wt{X}^*_s),X^*_s,\eta(P_s,X^*_s))\rt]\bigg\}\d s,
\end{aligned}
\end{cases}  	
\end{equation}
which is a coupled mean-field FBSDE, with the forward SDE  driven by a fractional Brownian motion. 
\begin{theorem}
Under the same conditions as in Theorem \ref{thm:suffi1}. There exists a unique solution $(X^*_t, P_t)\in S^2_{\mathbb{F}}([0,T];\R)\times S^2_{\mathbb{F}}([0,T];\R)$ of mean-field FBSDE (\ref{eq:necessary1})  for a small enough $T>0$. Here $S^2_{\mathbb{F}}([0,T];\R)$ denotes the space of real-valued $\mathbb{F}$-adapted continuous uniformly square integrable processes under the norm $\|\varphi\|_{S^2_{\mathbb{F}}([0,T];\R)}=\E\lt[\sup_{0\le s\le T}|\varphi_s|^2\rt]$.
\end{theorem}
\begin{proof}
	For any given $(x_t,p_t)\in S^2_{\mathbb{F}}([0,T];\R)\times S^2_{\mathbb{F}}([0,T];\R)$, we construct the following map $(X_t,P_t)=I(x_t,p_t)$:
	\begin{equation}
	\label{eq:necessary2}
	\begin{cases}
	\begin{aligned}
	X_t=&x+\sigma B^H_t+\int^t_0 b(\PP_{(x_s, \eta(p_s,x_s))},x_s, \eta(p_s,x_s))\d s,\\ 
	P_t =&\partial_x g(x_T,\PP_{x_T})+\wt{\E}\lt[\partial_\mu g(\wt{x}_T,\PP_{x_T},x_T)\rt]-\int^T_t \beta_s\d W_s\\
	&+\int^T_t\bigg\{\wt{\E}\lt[\wt{P}_s(\partial_{\mu}b)_1\lt(\PP_{(x_s,\eta(p_s,x_s))},x_s, \eta(\wt{p}_s,\wt{x}_s), {x}_s,\eta(p_s,x_s)\rt)\rt]\\&+ P_s\partial_xb\lt(\PP_{(x_s,\eta(p_s,x_s))}, x_s, \eta(p_s,x_s)\rt)+\partial_x f(\PP_{(x_s,\eta(p_s,x_s))},x_s,\eta(p_s,x_s))\\
	&+\wt{\E}\lt[(\partial_\mu f)_1(\PP_{(x_s,\eta(p_s,x_s))},x_s, \eta(\wt{p}_s,\wt{x}_s), {x}_s,\eta(p_s,x_s))\rt]\bigg\}\d s.
	\end{aligned}
	\end{cases}  	
	\end{equation}
	From the proof of Theorem \ref{thm:suffi1} we know that $0\le P\le C$ and the square integrability of $X$ can be derived from the linear growth of $b$ as in the proof of Theorem \ref{thm-gir-unique}. Hence $I$ maps from $S^2_{\mathbb{F}}([0,T];\R)\times S^2_{\mathbb{F}}([0,T];\R)$ to itself. Now we prove that it is a contracting map.  
	For $\lt(x^{1},p^{1}\rt)$ and $\lt(x^{2},p^{2}\rt)$ in $S^2_{\mathbb{F}}([0,T];\R)\times S^2_{\mathbb{F}}([0,T];\R)$,  from the Lipschitzianity of the functions $b$, the boundedness and Lipschitzianity of $\partial_x g$, $\partial_mu g$, $\partial_\mu b$, $\partial_x b$, $\partial_\mu f$ and $\partial_x f$, and the boundedness of $P^{1}$ and $P^{2}$, there exist a constant $C$ which does not depend on $T$, that
	\begin{equation}
	\begin{aligned}
    \E\lt[\sup_{0\le s\le t}|X^{1}_s-X^{2}_s|^2\rt]\le& C\int^t_0\lt(\E\lt[|x^{1}_s-x^{2}_s|^2\rt]+\E\lt[|p^{1}_s-p^{2}_s|^2\rt]\rt)\d s\\
    \le&Ct\lt(\E\lt[\sup_{0\le s\le t}|x^{1}_s-x^{2}_s|^2\rt]+\E\lt[\sup_{0\le s\le t}|p^{1}_s-p^{2}_s|^2\rt]\rt),
	\end{aligned}
	\end{equation}
	and 
		\begin{equation}
		\begin{aligned}
		\E\lt[\sup_{t\le s\le  T}|P^{1}_s-P^{2}_s|^2\rt]\le&
		C\E\lt[|x^{1}_T-x^{2}_T|^2\rt]+ C\int^T_t\lt(\E\lt[|x^{1}_s-x^{2}_s|^2]+\E[|p^{1}_s-p^{2}_s|^2\rt]\rt)\d s\\
		\le &C(T-t)\lt(\E\lt[\sup_{t\le s\le T}|x^{1}_s-x^{2}_s|^2\rt]+\E\lt[\sup_{t\le s\le T}|p^{1}_s-p^{2}_s|^2\rt]\rt).
		\end{aligned}
		\end{equation}
		Hence for a small enough $T$ such that $CT<\alpha<1$, we have 
		\begin{equation}
		\begin{aligned}
		&\E\lt[\sup_{0\le s\le  T}|X^{1}_s-X^{2}_s|^2+\sup_{0\le s\le  T}|P^{1}_s-P^{2}_s|^2\rt]\\
		\le&\alpha\E\lt[\sup_{0\le s\le T}|x^{1}_s-x^{2}_s|^2+\sup_{0\le T\le T}|p^{1}_s-p^{2}_s|^2\rt],
		\end{aligned}
		\end{equation}
		which means $I$ is a contracting map. Therefore, there exists a unique fixed point $(X^*,P)$, which is the solution for equation (\ref{eq:necessary1}).
\end{proof}
        With this solution $(X^*,P)$, we substitute into $\eta(X^*_t,P_t)$ and get $u^*_t$, which is a feedback sense optimal control, then due to Theorem \ref{thm:suffi1}, $(u^*_t,X^*_t)$ is the optimal pair for the stochastic control system.
\begin{remark}
        One can see that the coupled mean-field FBSDE (\ref{eq:necessary1}) is only in a special form. The general form of the fully coupled mean-field FBSDE driven by both a fractional Brownian motion and a classical Brownian motion is foreseen for a forthcoming paper. 
\end{remark}


\begin{thebibliography}{90}

\bibitem{BHOS}	
F.~Biagini, Y.~Hu, B.~$\O$ksendal and A.~Sulem. 
\newblock A stochastic maximum principle for processes driven by fractional Brownian motion. 
\newblock Stochastic Process. Appl.,  100 (2002), 233-253.
	
\bibitem{BHOZ}
F.~Biagini, Y.~Hu, B.~$\O$ksendal and T.~Zhang.
\newblock Stochastic Calculus for Fractional Brownian Motion and Applications.
\newblock Springer, 2008.	

\bibitem{Bu} R.~Buckdahn.
\newblock Anticipative Girsanov Transformations and Skorohod Stochastic Differential Equations.
\newblock Mem. Amer. Math. Soc., 111, 1994.

\bibitem{BDLP} R.~Buckdahn, B.~Djehiche, J.~Li and S.~Peng.
\newblock Mean-field backward stochastic differential equations: A limit approach.
\newblock Ann. Probab., 37 (2009), 1524-1565.

\bibitem{BuJi}
R.~Buckdahn and S.~Jing.
\newblock Peng's maximum principle for a stochastic control problem driven by a fractional and a standard Brownian motion.
\newblock Sci. China Math., 57 (2014), 2025-2042.

\bibitem{BLP}   R.~Buckdahn, J.~Li and S.~Peng.
\newblock Mean-field backward stochastic differential equations and related partial differential equations.
\newblock Stochastic Process. Appl.,  119 (2009), 3133-3154.
 
\bibitem{Car}
P.~Cardaliaguet.
\newblock Notes on mean field games.
\newblock Technical report, 2010. Available on his homepage: https://www.ceremade.dauphine.fr/~cardalia/MFG20130420.pdf

\bibitem{CaDe}
R.~Carmona and F.~Delarue.
\newblock Forward–backward stochastic differential equations and controlled McKean–Vlasov dynamics.
\newblock Ann. Probab., 43 (2015), 2647-2700.

\bibitem{DHP}
T.~Duncan, Y.~Hu and B.~Pasik-Duncan.
\newblock Stochastic calculus for fractional Brownian motion I. Theory. 
\newblock SIAM J. Control Optim., 38 (2000), 582-612.

\bibitem{HHS}
Y.~Han, Y.~Hu and J.~Song. 
\newblock Maximum principle for general controlled systems driven by fractional Brownian motions. 
\newblock Appl. Math. Optim., 67 (2013), 279-322.

\bibitem{Hu}
Y.~Hu.
\newblock Integral Transformations and Anticipative Calculus for Fractional Brownian Motions.
\newblock Mem. Amer. Math. Soc., 2005.

\bibitem{HuZh}
Y.~Hu and X.~Zhou. Stochastic control for linear systems driven by fractional noises. SIAM J. Control Optim.,  43 (2005), 2245-2277.

\bibitem{JL}
S.~Jing and J.~A.~Le\'on.
\newblock Semilinear backward doubly stochastic differential equations and SPDEs driven by fractional Brownian motion with Hurst parameter in $(0, 1/2)$.
\newblock Bull. Sci. Math., 135 (2011), 896-935.

\bibitem{Ka1}
 M.~Kac. 
\newblock Foundations of kinetic theory. 
\newblock Proceedings of the 3rd Berkeley Symposium
on Mathematical Statistics and Probability, 3 (1956), 171-197.

\bibitem{Ka2} 
M.~Kac. 
\newblock Probability and Related Topics in the Physical Sciences.
\newblock Interscience Publishers, New York, 1958.

\bibitem{LaLi}
J.~M.~Lasry and P.~L.~Lions. 
\newblock Mean field games. 
\newblock Jpn. J. Math., 2 (2007), 229-260.

\bibitem{Mi}
Y.~Mishura. 
\newblock Stochastic Calculus for Fractional Brownian Motion and Related Processes. 
\newblock Springer, 2008.

\bibitem{Nu}
D.~Nualart. 
\newblock The Malliavin Calculus and Related Topics. 
\newblock 2nd Edition, Springer-Verlag, Heidelberg, 2006.


\bibitem{PeWu}
\newblock S.~Peng and Z.~Wu. 
\newblock Fully coupled forward-backward stochastic differential equations and applications to optimal
control. 
\newblock SIAM J. Control Optim., 37 (1999), 825-843.

\bibitem{Sa} 
S.~G.~Samko, A.~A.~Kilbas and O.~I.~Marichev.
\newblock Fractional Integrals and Derivatives: Theory and Applications.
\newblock Gordon and Breach Science Publishers, 1987.

\end{thebibliography}
\end{document}